%% file: BET-Nov2021.tex
\documentclass[11pt]{amsart}
\usepackage{amssymb,bm}
\usepackage{amsmath}
\usepackage{amsfonts}
\usepackage{amsthm}
\usepackage{color}
\usepackage{graphics,graphicx,url}
\usepackage{enumerate}
\usepackage{comment}
\usepackage{subcaption}
\usepackage{thmtools}
\usepackage{thm-restate}

\usepackage{centernot}

\usepackage{breqn} 

\usepackage[a4paper]{geometry}
\usepackage{xcolor}

\usepackage{tikz}
\usetikzlibrary{calc,decorations.pathmorphing,patterns,arrows,decorations.markings,positioning}
\usetikzlibrary{automata,chains,fit,shapes}

\usepackage{graphicx,pdflscape,rotating,float}
\usepackage{enumitem}
\newcommand{\bi}{\begin{itemize}}
\newcommand{\ei}{\end{itemize}}
\newcommand{\be}{\begin{enumerate}}
\newcommand{\ee}{\end{enumerate}}
\newcommand{\bc}{\begin{center}}
\newcommand{\ec}{\end{center}}
\newcommand{\bt}{\begin{tabular}}
\newcommand{\et}{\end{tabular}}
\newcommand{\ba}{\begin{array}}
\newcommand{\ea}{\end{array}}

\newcommand{\phiGS}{\phi_{G,S}}

\newcommand{\cA}{c}
\newcommand{\cB}{d}
\newcommand{\cC}{\varsigma}
\newcommand{\nD}{n_0}
\newcommand{\preceqD}{\preceq_2}
\newcommand{\approxD}{\approx_2}
\newcommand{\preceqF}{\preceq_1}
\newcommand{\npreceqF}{\npreceq_1}
\newcommand{\approxF}{\approx_1}

\newcommand\BaumGroup{{\mathfrak G}_{G,H,S}}
\newcommand\Olshan{Ol'shanskii}

\newenvironment{proof*}[1]
  {\renewcommand{\proofname}{#1}
   \begin{proof}}
  {\end{proof}}

\newtheorem{theorem}{Theorem}[section]

\newtheorem{conjecturex}{\bfseries Conjecture}
\newtheorem{lemma}[theorem]{Lemma}

\newtheorem{proposition}[theorem]{Proposition}

\theoremstyle{definition}
\newtheorem{remark}[theorem]{Remark}
\newtheorem{definition}[theorem]{Definition}
\newtheorem{example}[theorem]{Example}

\usepackage[normalem]{ulem}

\newcommand\N{\mathbb N}

\newcommand\R{\mathbb R}
\newcommand\Rplus{\mathbb R^+}

\newcommand\dom{\mathrm{dom}}

\newcommand\distfun{Cayley distance function}

\newcommand\F{\mathcal F}

\newcommand{\ii}{\mathfrak{i}}

\renewcommand{\geq}{\geqslant} \renewcommand{\leq}{\leqslant}  \renewcommand{\le}{\leqslant}

\title{ On the geometry of Cayley automatic groups}

\author[D. Berdinsky]{Dmitry Berdinsky\textsuperscript{a,b}}
\address{\textsuperscript{a}Department of Mathematics,
		Faculty of Science,
		Mahidol
		University, Bangkok, 10400, Thailand
		\textsuperscript{b}Centre of Excellence in Mathematics,
		Commission on Higher Education, Bangkok, 10400, Thailand}
\email{berdinsky@gmail.com}

\author[M. Elder]{Murray Elder}\thanks{The second author  acknowledges support from  Australian Research Council grant DP160100486}
\address{School of Mathematical and Physical Sciences, University of Technology Sydney, Ultimo, NSW 2007, Australia}
\email{murray.elder@uts.edu.au}

\author[J. Taback]{Jennifer Taback}\thanks{The third author acknowledges support from Simons Foundation grant 31736 to Bowdoin College.}
\address{Department of Mathematics,
Bowdoin College, 8600 College Station, Brunswick, ME 04011, USA} \email{jtaback@bowdoin.edu}

\date{\today}



\begin{document}

\begin{abstract}
In contrast to being automatic, being Cayley automatic {\em a priori} has no geometric consequences. 
Specifically, Cayley graphs of automatic groups enjoy a fellow traveler property. 
Here we study a 
distance function introduced by the first author and  Trakuldit which aims to 
measure how far a  Cayley automatic group is from being automatic, in terms of how badly the Cayley graph  fails the fellow traveler property.
The first author and  Trakuldit showed that if it fails by at most a constant amount, then the group is in fact automatic. 
In this article we show that for a large class of non-automatic Cayley automatic groups this function is 
bounded below by a linear function in a precise sense defined herein. In fact,  for all  Cayley automatic groups which have super-quadratic Dehn function, or which are not finitely presented, we can construct a non-decreasing function which (1) depends only on the group and (2)  bounds from below the distance function for any Cayley automatic structure on the group.
\end{abstract}

\maketitle

\section{Introduction}
\label{sec:intro}

Cayley automatic groups generalize the class of automatic groups while retaining their key {\em algorithmic} 
properties.  
Namely, the
word problem in a Cayley automatic group is decidable in quadratic time,
regular normal forms for group elements can be computed in quadratic time, and
the first order theory for a (directed, labeled) Cayley graph of a Cayley automatic group is decidable. 
Their history traces back to   S{\'e}nizergues who observed that the (standard) Cayley graph for the integral Heisenberg group is FA-presentable,  with 
a proof
first appearing in 
\cite{BlumensathG}
in 2004
 (see in particular page 651), 
and the concept was brought to the attention of combinatorial/geometric group theorists by  Kharlampovich, Khoussainov
and Miasnikov in \cite{KKMjournal}.

The family of Cayley automatic groups is much broader than that of automatic groups, as it  includes, for example, all finitely generated nilpotent
groups of nilpotency class
two \cite{KKMjournal}, the Baumslag-Solitar groups \cite{BerK-BS,KKMjournal},
higher rank lamplighter groups \cite{Taback18}, and restricted wreath products of the form $G\wr H$ where $G$ is Cayley automatic and $H$ is (virtually) infinite cyclic \cite{BETwreath,berdkhouss15}.

The existence of a Cayley automatic structure for a group $G$ appears to impose no restrictions on its geometry.  This differs from the existence of an automatic structure; if a group $G$ admits an automatic structure then the Cayley graph with respect to any finite generating set $S$ enjoys the so-called fellow traveler property.
This geometric condition requires that the normal form representatives for a pair of group elements at distance 1 in the Cayley graph $\Gamma(G,S)$ remain a uniformly bounded distance apart in this graph.

The goal of this paper is to explore the geometry of the Cayley graph of a Cayley automatic group, and in particular, to understand an analogue of the fellow traveler property for these groups.
A Cayley automatic group differs from an automatic group in that the normal form for group elements is defined over a finite symbol alphabet rather than a set $S$ of generators.  However, without loss of generality we can take these symbols to be additional generators, renaming the larger generating set $S$, and for $g \in G$ obtain a normal form which describes a path in the Cayley graph $\Gamma(G,S)$ but (most likely) not a path to the vertex labeled  $g$.  
In lieu of a fellow traveler property, we investigate the distance in $\Gamma(G,S)$ between the vertex labeled $g$ and the endpoint of this path.

To quantify how far a Cayley automatic structure is from being automatic, we follow the first author and Trakuldit in  \cite{measuringcloseness} and define the Cayley distance function $f_{\psi}$ for a given Cayley automatic structure $\psi$, where $f_{\psi}(n)$ is the maximum distance between a normal form word representing $g \in G$ and the vertex labeled $g$, over all 
normal forms of word length at most $n$.
In \cite{measuringcloseness} it is shown that $G$ is automatic if and only if this function is equivalent to a constant function, in a notion of equivalence defined below.
Thus for Cayley automatic groups which are not automatic this function is always unbounded and non-decreasing.
This motivates our investigation of when the \distfun\ might be bounded below by a non-constant function, quantitatively separating $G$ from the class of automatic groups.
However, the possibility exists that a group may admit a sequence of Cayley automatic structures for which the corresponding sequence of \distfun s limit to a constant function, but never contains the  constant function.

In this paper we prove that such limiting behaviour is not possible in any group which is not finitely presented, or in any finitely presented group that has a  super-quadratic Dehn function, as given in Definition~\ref{defn:strong-super}.  In each case, we  construct a concrete unbounded function depending only on the group, so that the \distfun\ for any Cayley automatic structure on the group is bounded below by this function, up to equivalence.
We say that a Cayley automatic
group $G$ is {\em $f$-separated} if the \distfun\ with respect to any Cayley automatic structure on $G$ is bounded below by a function in the equivalence class of $f$ (Definition~\ref{def:f-separated}).

Let $\ii$ denote the function $\ii(n)=n$ on some domain $[N,\infty)$.
Super-quadratic and strongly-super-polynomial functions, referred to in Theorem~\ref{thmA:fp} below, are introduced in Definition~\ref{defn:strong-super}.
We prove the following.
\begin{restatable}[Finitely presented groups]{theoremx}{ThmA}
\label{thmA:fp}
If $G$ is a finitely presented Cayley automatic group with super-quadratic Dehn function,
 then there exists an unbounded function $\phi$ depending only on $G$ so that $G$ is $\phi$-separated.
 Furthermore, if $G$ has strongly-super-polynomial Dehn function, then  $G$ is $\ii$-separated.
\end{restatable}

 The analogous theorem for non-finitely presented groups is as follows.
A non-finitely presented group is {\em dense} if its irreducible relators have lengths which are ``dense" in the natural numbers; see \S\!~\ref{sec:dense} for a  precise definition.  Wreath products are the prototypical examples of dense groups.

\begin{restatable}[Non-finitely presented groups]{theoremx}{ThmB}
\label{thmB:nfp}
If $G$ is a Cayley automatic group which is not finitely presented, then there is a non-decreasing step function $\phi$ depending only on $G$
that is  linear for infinitely many values,
so that $G$ is $\phi$-separated.
Furthermore, if $G$ is dense then $G$ is $\ii$-separated.
\end{restatable}

We conjecture that  for every Cayley automatic group that is not automatic, the  distance function with respect to every Cayley automatic structure on the group
 is bounded below by a linear function, which is equivalent to being $\ii$-separated.
\begin{conjecturex}\label{conj1}
Let $G$ be a  Cayley automatic group. Then $G$ is either automatic or
 $\ii$-separated.
\end{conjecturex}

Our results provide  support for this conjecture by exhibiting lower bounds for \distfun s  for all non-automatic Cayley automatic groups whose Dehn function is super-quadratic or are not finitely presented.  However, these bounds are not always equivalent to $\ii$.

While we believe the conjecture to be true, two groups for which this linear lower bound is not obvious to us are the following.
\bi[itemsep=5pt]
\item The higher Heisenberg groups  $H_{2k+1}$ for $k\geq 2$. 
 These are nilpotent of step 2  so they are Cayley automatic by  \cite[Theorem 12.4]{KKMjournal}. Since the only nilpotent automatic groups are virtually abelian \cite{Epsteinbook}, they are not automatic.  It is proved in \cite{MR1616147,MR1253544,MR1698761} that their Dehn function is quadratic.

\item The higher rank lamplighter groups, or {\em Diestel-Leader groups}, proven to be Cayley automatic by B\'{e}rub\'{e}, Palnitkar and   the third author in \cite{Taback18}. One can show that the Cayley automatic structure constructed in  \cite{Taback18} has Cayley distance function equivalent to the identity function. These groups are not of type FP$_\infty$ \cite{BartholdiNW}, hence not automatic.
See \cite{Taback18} for a discussion explaining why their Dehn functions are quadratic.

\ei

The paper is organised as follows. In Section~\ref{sec:Auto-CGA} we review automatic and Cayley automatic groups,  define the \distfun\ for a Cayley automatic structure, and finish with a short discussion of Dehn functions. In Section~\ref{sec:FP-dehn} we prove Proposition~\ref{prop:dehnbound}, which relates the \distfun\ to the Dehn function for a finitely presented Cayley automatic group. In  Section~\ref{sec:FP} we define super-quadratic, super-polynomial and strongly-super-polynomial functions and prove Theorem~\ref{thmA:fp}. 
We then turn to non-finitely presented Cayley automatic groups. We introduce the notion of a dense group in  Section~\ref{sec:dense}, then prove Theorem~\ref{thmB:nfp} in Section~\ref{sec:thmB}. We include additional information about strongly-superpolynomial functions in Appendix~\ref{appendix:super-strong}.

\section{Automatic and Cayley automatic groups}
\label{sec:Auto-CGA}

We assume that the reader is familiar with the notions of regular languages, finite automata and  multi-tape synchronous automata.  For more details, we refer the reader to \cite{Epsteinbook}. We say a language $L\subseteq (X^*)^n$ is {\em regular}  if it is accepted by a synchronous $n$-tape automaton where $n\in\N$ and $X$ is a finite set, or {\em alphabet}.

For any group $G$ with finite symmetric generating set $S=S^{-1}$, let $\pi\colon S^*\to G$ denote the canonical projection map. For $w\in S^*$ let $|w|_S$ denote the length of $w$ as a word in the free monoid $S^*$. 

For any $M,N\in \N$, let $[M,N]=\{n\in\N \mid M\leq n\leq N\}$ and $[M,\infty)=\{n\in\N \mid n\geq M\}$.

\subsection{Automatic and Cayley automatic groups}

We define automatic and Cayley automatic groups, and provide some standard lemmas on the invariance of the  Cayley automatic structure under change of generating set.

\begin{definition}
\label{def:aut}
An {\em automatic structure} for a group $G$ is a pair $(S,L)$ where
\be
\item $S$ is a finite symmetric generating set for $G$;
\item $L\subseteq S^*$ is a regular language;
\item  $\pi|_L \colon L \rightarrow G$ is
a bijection;
\item for each $a \in S$ the binary relation
$$R_a = \{(u,v) \in L \times L \mid \pi(u)a=_G\pi(v)\} \subseteq S^* \times S^*$$
is regular, that is, recognized by a two-tape synchronous automaton.
\ee
A group is called {\em automatic} if it has an automatic structure with respect to some finite generating set.
\end{definition} 
It is a standard result, see, for example \cite[Theorem 2.4.1]{Epsteinbook}, that if $G$ is automatic  then $G$ has an automatic structure with respect to any finite generating set.

Cayley automatic groups were introduced in \cite{KKMjournal} with the motivation of
allowing the language $L$ of normal forms representing group elements to be defined over a
symbol alphabet $\Lambda$ rather than a generating set $S$ for $G$.

\begin{definition}
\label{def:Caut}
A {\em Cayley automatic structure} for a group  $G$ is a 4-tuple  $(S,\Lambda, L,\psi)$ where
\be
\item $S$ is a finite symmetric generating set for $G$;
\item  $\Lambda$ is an alphabet and
 $L \subseteq \Lambda^*$  is a regular
language;
\item
$\psi\colon L \rightarrow G$ is a bijection;
\item
for each $a \in S$
the binary relation
$$R_a = \{(u,v) \in L \times L \,
|\,\psi(u)a=_G\psi(v)\} \subseteq \Lambda^* \times \Lambda^*$$
is regular, that is, recognized by
a two-tape synchronous automaton. \ee

A group is called {\em Cayley automatic} if it has a Cayley automatic structure   $(S,\Lambda, L,\psi)$  with respect to some finite generating set $S$.\end{definition}
As for automatic groups, if $G$ has a  Cayley automatic structure    $(S,\Lambda, L,\psi)$
and $Y$ is another finite generating set for $G$, then there exists a Cayley automatic structure    $(Y,\Lambda_Y, L_Y,\psi_Y)$ for $G$. See \cite[Theorem 6.9]{KKMjournal} for a proof of this fact; we sharpen this in Proposition~\ref{prop:modifyingCAstructure} below.

Note that a Cayley automatic structure $(S,S,L,\pi|_L)$ for $G$, that is, one in which the symbol alphabet is in fact a generating set, and the natural projection gives a bijection from $L$ to $G$, is simply an automatic structure for $G$.

{\em A priori} the symbol alphabet $\Lambda$ has no relation to a generating set for $G$. 
However  it is straightforward to show that if $(S,\Lambda,L, \psi)$ is a Cayley automatic structure for  $G$, then there exists another Cayley automatic structure  $(S',S',L, \psi)$ for $G$ where $S'=\Lambda\cup\Lambda^{-1}\cup S$, and so we can always associate a Cayley graph with respect to a generating set which includes symbol letters from the Cayley automatic structure.  This is proven in \cite{measuringcloseness} and in Proposition~\ref{prop:modifyingCAstructure} below.
In this case, a word $w\in L$ labels a path from $1_G$ to $\pi(w)$ in the Cayley graph $\Gamma(G,S)$.
It is crucial to note that in general $\pi(w)\neq \psi(w)$.

\begin{definition}
\label{def:h}
Let $(G,S)$ be a group with Cayley automatic structure $(S,S, L,\psi)$.  The {\em \distfun}  corresponding to $\psi$ is defined to be
$$ h_{S,\psi}(n) = \max \{ d_S (\pi (w),\psi(w)) \,| \, w \in L^{\leqslant n}\}$$
where $d_S$ is the word  metric on $G$ with respect to $S$ and
$$L^{\leqslant n} = \{ w \in L \, |\, |w|\leqslant n\}.$$
\end{definition}

Let $ \F$ be the following set of non–decreasing functions:
\[ \F=\{f \colon [N,\infty)\to \Rplus \mid N\in\N\wedge \forall n (n\in \mathrm{dom} f \Rightarrow f(n)\leq f(n+1))\}.\]
Note that if
$G$ is a group with Cayley automatic structure $(S,S,L,\psi)$ and \distfun\ $h_{S,\psi}$, then  $h_{S,\psi}\in \F$.

We introduce the following partial order on $\F$.

\begin{definition}\label{defn:equivF} Let $f,g \in   \F$. We say that $g \preceqF f$ if there exist positive integers $K,M$ and $N$ such that $[N,\infty) \subseteq
\mathrm{dom}\,g\cap \mathrm{dom}\, f$ and $g(n)\leq Kf(Mn)$ for every integer $n\geq N$.
We say that $g \approxF f$ if $g \preceqF f$ and $f \preceqF g$.
\end{definition}
The subscript in Definition~\ref{defn:equivF} serves to distinguish this equivalence from the equivalence on Dehn functions discussed in \S\!~\ref{subsec:Dehn}.  It is clear from the definition that both $\preceqF$ and $\approxF$ are  transitive.

\begin{definition}
\label{def:f-separated}
Let $f,g\in \F$, and let $G$ be a Cayley automatic
group.
 We say that $g$ is {\em $f$-separated} if $f\preceqF g$, and that 
 $G$ is {\em $f$-separated} if for every 
Cayley automatic structure $(S,S,L,\psi)$ on $G$ with distance function $h_{S,\psi}$, $h_{S,\psi}$ is $f$-separated.
\end{definition}
Note that since $\preceqF,\approxF$ are both transitive, if $G$ is $f$-separated and $h\approxF f$ then $G$ is $h$-separated also.

Let $\mathbf z$ denote the zero function $\mathbf z (n)=0$ on
some domain $[N,\infty)$, and $\ii$ the function $\ii(n)=n$ on some domain $[N,\infty)$.
We note that if $f\in\F$ and $f(n)=0$ for infinitely many values of $n\in\dom f$ then $f=\mathbf z$ on its domain, because $f\in\F$ is non-decreasing.

The next lemma will be used repeatedly in the proofs below, and is a fact about certain types of functions which are easily seen to be related or equivalent under the definition above.

\begin{lemma}
\label{lemma:lose_the_constant}
Let $A,B,C,D\in\R$ with $A, D\geq 1$ and $B,C\geq 0$.
Let $f,g \in \F$ with $f(n) \leq D g(An+B)+C$.
If $g\neq \mathbf z$, then $f(n) \preceqF g(n)$.  Moreover, if $h(n)=Df(An+B)+C$
and $f\neq \mathbf z$,
then  $h \in \F$ and $h\approxF f$.
\end{lemma}

\begin{proof}
If $g$ is bounded, so $g(n)\leq E$ for all $n\in \dom\,g$ for some fixed constant $E$, then 
$f$ is bounded as well.
Since $g(n)=0$ for at most finitely many values of $n\in\dom g$, we have $f(n) \preceqF g(n)$ (possibly increasing $N_0$).
If $f$ is bounded  and  $f(n)=0$ for at most finitely many values of $n\in\dom f$,  it follows immediately that $h\approxF f$.

For the remainder of the proof, we assume that $g$ is not bounded.
There is a constant $N_0$ so that for $n \geqslant N_0$ we have $An \geqslant B$.  As $g \in \F$, it follows that for $n \geqslant N_0$ we have $D g(An+B) +C\leq D g(2An)+C$.
Since $g$ is not bounded, there is a constant $N_1$ so that for $n \geq N_1$ we have $g(2An) \geqslant C$.
Then for $n \geq \max(N_0,N_1)$ we have $f(n) \leq D g(An+B)+C \leq D g(2An) +C \leq (D+1) g(2An)$ and thus $f \preceqF g$.

Letting $g=f \in \F$ the above reasoning shows that
$h \preceqF f$. As it is clear that
$h(n)=Df(An+B)+C \in \F$ and $f \preceqF h$,
it follows that $f \approxF h$, as desired.
\end{proof}

Note that $\approxF$ defines an equivalence relation
on the set $\F$ and $\preceqF$ then gives a partial ordering on the resulting set of equivalence classes.
The poset of equivalence classes of elements of
$\F$ has a minimal element $[\mathbf z]$.
It follows from the previous lemma that all bounded
functions  $f\in\F$ for which $f(n)=0$ for at most finitely many values of $n\in\dom f$
are in the same equivalence class.
Furthermore, every $f \in \F$ can be compared to a constant function. In contrast, we show that the partial ordering is not a linear ordering, that is, there are functions in $\F$ which cannot be compared to the identity function $\ii$ under $\preceqF$.

\begin{lemma}\label{lemma:comp_to_constant}
Let $c \in \N$ and $g: \N \rightarrow \N$ the constant function $g(n)=c$.  Let $f \in  \F$ be any function.  Then either $f\preceqF g$ or $g\preceqF f$.
\end{lemma}

\begin{proof}
If
there is some $D\in \Rplus$ so that $f(n)\leq D=\left(\frac{D}{c}\right)c=\left(\frac{D}{c}\right)g(n)$ for all $n\in\dom f$ then $f\preceqF g$.
If not, then for all $D\in \Rplus$ there is an integer $N_D\in\dom f$ so that $f(n)>D$ for $n>N_D$. In particular, there is an integer $N_c \in \dom f$ so that $f(n)>c=g(n)$ for all $n\in \dom f\cap [N_c,\infty)$.  Thus $g\preceqF f$.
\end{proof}

Lemma~\ref{lem:example-incompariable} demonstrates that not every function $f\in \F$ is $\ii$-separated.

\begin{lemma}\label{lem:example-incompariable}
There exists a function $f\in  \F$  so that $\ii\npreceqF f$ and $f\npreceqF \ii$.
\end{lemma}
\begin{proof}
 Let $n_0 =2$ and define the infinite sequence of integers $n_{i+1}=n_i^2 = 2^{2^i}$.
 Consider the step function $f\colon\N\to \Rplus$
 defined by
\[f(x)= \left\{\begin{array}{lll}
n_{2i} &n_{2i} \leq x < n_{2i+1},\\
n_{2i+2} &n_{2i+1} \leq x < n_{2i+2}.\\
\end{array}  \right.\]

Suppose $f\preceqF \ii$. Then $\exists N_0, K,M$ so that $f(x)\leq K\ii(Mx)= KMx$ for all $x\geq N_0$.
However,
\[ f(n_{2i+1})=n_{2i+2}=n_{2i+1}^2\leq KM n_{2i+1}\]
which implies that
$n_{2i+1}\leq KM $ for sufficiently large $i$, a contradiction.
Thus $f\npreceqF \ii$.

Conversely suppose $\ii\preceqF f$. Then $\exists N_0, K,M$ so that $x\leq Kf(Mx)$ for all $x\geq N_0$. This means $\frac{s}{M}\leq Kf(s)$ for all $s=Mx\geq MN_0$, which
 implies that $M\lfloor\frac{s}{M}\rfloor\leq KMf(s)$ for all $s\geq MN_0$.
However,
\[M \left\lfloor\frac{n_{2i}^2-1}{M} \right\rfloor=
M \left\lfloor \frac{n_{2i+1}-1}{M} \right\rfloor \leq KMf(n_{2i+1}-1)=KMn_{2i}
\]
and thus $\left\lfloor \frac{n_{2i}^2-1}{M} \right\rfloor \leq Kn_{2i}$.
Therefore,
$\frac{n_{2i}^2-1}{M} \leqslant K n_{2i} + 1$,
so $n_{2i} \leqslant KM + \frac{M+1}{n_{2i}}$
which is a contradiction for sufficiently large $i$.
Thus $\ii\npreceqF f$.
\end{proof}

\subsection{Invariance under change of generating set and change of structure}

Here we describe how robust both the Cayley automatic structure and the function $h_{S,\psi}$ are to, respectively, change in generating set and change in structure.
First we recall the following standard fact.
\begin{lemma}\label{lem:2tape}
Let $G$ be a Cayley automatic group and $S$ a  finite symmetric generating set for $G$.
Let $(S,\Lambda,L,\psi)$ be a Cayley automatic structure for $G$.
Then for any $w\in S^*$,
\[L_w=\{(u,v)\in L^2\mid \psi(v)=_G\psi(u)w\}\] is regular.
\end{lemma}
\begin{proof}
Let $w=s_1\dots s_n$ where $s_i\in S$ for $1 \leq i \leq n$.
As $(S,\Lambda,L,\psi)$ is a Cayley automatic structure for $G$, for each $s \in S$ there is a synchronous 2-tape automaton $\texttt{M}_{s}$  which accepts the language $$L(\texttt{M}_{s})=\{(u,v)\in L^2\mid \psi(v)=_G\psi(u)s\}.$$
Let $\texttt{M}_i'$ be a synchronous  $(n+1)$-tape automaton accepting  \[(z_0, \dots, z_{i-1}, u, v, z_{i+2},\dots, z_n)\] where $z_j\in \Lambda^*, j\in[0,i-1]\cup[i+1,n]$, $u,v\in L$ and  $\psi(v)=\psi(u)s_i$.
We construct $\texttt{M}_i'$ from $\texttt{M}_{s_i}$ by replacing each edge labeled $(a,b)\in \Lambda^2$ by the finite number of edges
labeled $(x_0,\dots x_{i-1},a,b,x_{i+2},\dots ,x_n)\in \Lambda^{n+1}$ for all possible choices of $x_j \in \Lambda$
where $0\leq j \leq i-1$ and $i+2 \leq j \leq n$.

Then \[L_w=\bigcap_{i=1}^n  L(\texttt{M}_i')\] which is regular since this is a finite intersection, then apply a homomorphism to project onto the first and last factors.
\end{proof}

\begin{proposition}\label{prop:modifyingCAstructure}
Let $G$ be a Cayley automatic group and $S$ a  finite symmetric generating set for $G$.
\be\item If  $(S,\Lambda,L, \psi)$ is a Cayley automatic structure for  $G$, then so is $(S',S',L, \psi)$ where $S'=\Lambda\cup\Lambda^{-1}\cup S$

\item If $(S,S,L, \psi)$ is a Cayley automatic structure for  $G$ with \distfun\ $h_{S,\psi}$, and $Y$ is a   finite symmetric generating set for $G$,  then there exists a language $L'\subseteq Y^*$ and a bijection $\psi'\colon L'\to G$ so that
 $( Y, Y,L',\psi')$ is a Cayley automatic structure for  $G$ with \distfun\  $h_{ Y,\psi'}\approxF h_{S,\psi}$.
\ee
 \end{proposition}

\begin{proof}
\textit{(1)}  Suppose $\langle S\mid R\rangle$ is a presentation for $G$.
For each $a\in \Lambda$ choose an element $g_a\in G$, and  choose a word $u_a\in S^*$ with $\pi(u_a)=g_a$.
Note that this choice is arbitrary; the element $g_a$ corresponding to the symbol letter $a$ could be any group element.
Let $\Lambda^{-1}$ be the disjoint set $\{a^{-1}\mid a\in \Lambda\}$; we will not use these letters, but include them to ensure our new generating set is symmetric.
Since $\Lambda$ is finite, there is a bound on the length of all $u_a$ words.
We have  $S'=\Lambda\cup\Lambda^{-1}\cup S$ and
 $G$ is presented by $\langle S'\mid R\cup\{a=u_a\mid a\in\Lambda\}\rangle$.

With this new generating set, we have the same language $L$  which is regular, and the map $\psi\colon L\to G$.
For each $s\in S$ there is an automaton $\texttt{M}_s$ recognizing multiplication by $s$, and it follows from Lemma~\ref{lem:2tape} that there is an analogous 2-tape automaton $\texttt{M}_a$ for each $a\in \Lambda^{\pm 1}$.

\medskip
\noindent
\textit{(2)}
 We have $L\subseteq S^*$ is a regular language in bijective correspondence with $G$. For each $s\in S$, choose a word $u_s\in Y^*$ with $s=_Gu_s$ and for each $y\in Y$, choose a word $v_y\in S^*$ with $y=_Gv_y$. Let $M_1=\max\{|v_y|_S\mid y\in Y\}$ and
 $M_2=\max\{|u_s|_Y\mid s\in S\}$.

 Let the monoid homomorphism $\rho\colon S^*\to Y^*$ be defined by $\rho(s)=u_s$. It follows that $L'=\rho(L)$ is a regular language in bijection with $G$, where
 $\psi'\colon L'\to G$  defined by $\psi'=\psi\circ\left(\rho|_L\right)^{-1}$ is a bijection.
 Note that for all $w\in L$ we have $\pi(w)=\pi(\rho(w))$ and $\psi(w)=\psi'(\rho(w))$.
 For each $w\in L^{\leq n}$ we claim that
 \begin{equation}\label{eqn:inequality}
d_S\left(\pi(w), \psi(w)\right)\leq
 M_1h_{Y,\psi'}\left(M_2n\right)
 \end{equation}
 To see this, we argue as follows.
 \begin{itemize}
     \item Under $\rho$, the path labeled $w$ from $1_G$ to $\pi(w)$ in $\Gamma(G,S)$ is mapped to a path labeled $\rho(w)$ from $1_G$ to $\pi(\rho(w))=\pi(w)$ in $\Gamma(G,Y)$, and this path has length at most $M_2n$, replacing each letter $s$ of the path by $u_s$.   See  Figure~\ref{fig:quasiIsom}.
     \item By definition, the distance from $\pi(\rho(w))$ to $\psi'(\rho(w))$ in $\Gamma(G,Y)$ is at most
$h_{Y,\psi'}(M_2n)$ since this is the maximum such distance over all possible words in $(L')^{\leq M_2n}$.
\item Then in $\Gamma(G,Y)$ we have a path from $\pi(w)$ to $\psi(w)$ of length at most $h_{Y,\psi'}(M_2n)$ in the letters from $Y$; call it $\gamma$. Replacing each of these letters $y$ by $u_y$ we obtain a path $\rho^{-1}(\gamma)$ in $\Gamma(G,S)$ from $\pi(w)$ to $\psi(w)$ of length at most $M_1h_{Y,\psi'}(M_2n)$.
 \end{itemize}

\begin{figure}[h!]
\begin{subfigure}{.45\textwidth}
  \centering
 \input{qi1}
  \caption{In $\Gamma(G,S)$}
  \label{fig:sub-first}
\end{subfigure}
\begin{subfigure}{.45\textwidth}
  \centering
\input{qi2}
  \caption{In $\Gamma(G,Y)$}
  \label{fig:sub-second}
\end{subfigure}
\caption{Drawing  $w\in L$  and $\rho(w)\in L'$ in each  Cayley graph.}
\label{fig:quasiIsom}
\end{figure}
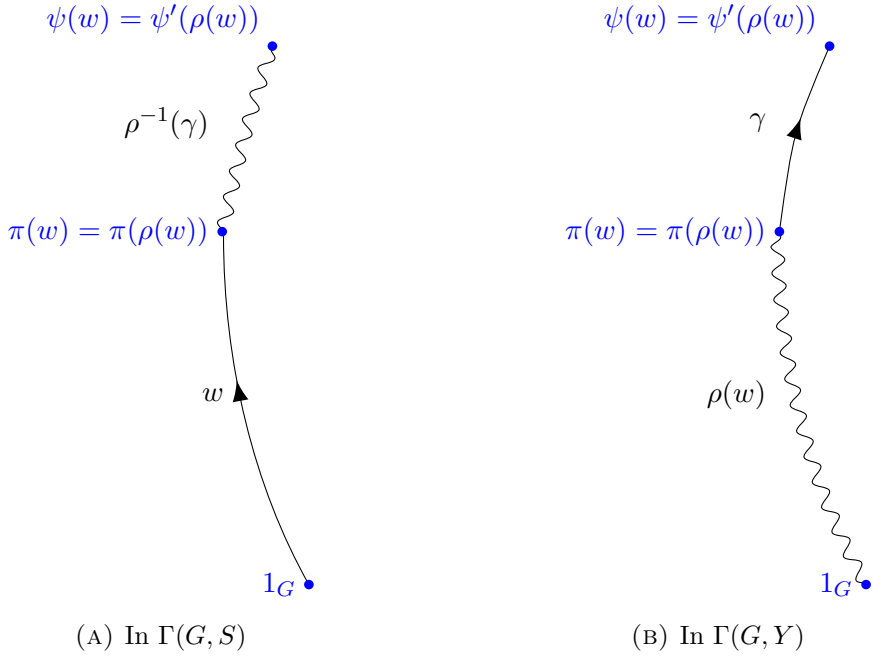

Since Equation~(\ref{eqn:inequality}) is true for all $w\in L^{\leq n}$, it follows that
 $h_{S,\psi}\preceqF h_{Y,\psi'} $.

Similarly
 for each $w'=\rho(w)\in (L')^{\leq n}$  the same argument shows that \[d_Y\left(\pi(w'), \psi'(w')\right)
 \leq
 M_2h_{S,\psi}\left(M_1n\right)
 \] as $\pi(w')=\pi(w)$ and $\psi(w)=\psi'(w')$ are the same vertices.

 Thus $h_{Y,\psi'}\preceqF h_{S,\psi}$ and it follows that  $h_{Y,\psi'}\approxF h_{S,\psi}$
\end{proof}

\begin{remark}
Note that a given group may admit many different Cayley automatic structures whose \distfun s are not equivalent under $\preceqF$.  Part (2) of Proposition~\ref{prop:modifyingCAstructure} proves that given one Cayley automatic structure for a group $G$ with respect to a generating set $S$,  we can create a new Cayley automatic structure for $G$ over a generating set $Y$ so that both \distfun s are equivalent under $\preceqF$.
\end{remark}

\subsection{Dehn functions}\label{subsec:Dehn}

Let $\mathcal P=\langle X\mid  R\rangle $ be a finite presentation of a group $G$ and $F_X$ the free group on $X$. If $w\in F_X$ is equal to $1_G$ in $G$, then
there exist $N\in \N$, $r_i\in R$ and $u_i\in F_X$ such that
\[w_{F_X}=\prod_{i=1}^N  u_i^{-1}r_i^{\epsilon_i}u_i
\]
We define the {\em area} of $w$,  denoted $A_{\mathcal P}(w)$,  to be the minimal $N\in \N$ so that $w$ has such an expression.

\begin{definition}
The {\em Dehn function} of a presentation is the function  $\delta_{\mathcal P}\colon \N\to\N$ given by
$$
\delta_{\mathcal P}(n)=\max\{A_{\mathcal P}(w)\mid  w\in F_X, w=_G1_G, |w|\le n \}
$$
\end{definition}

Note that if $f$ is a Dehn function then $f\in \F$.
It is standard to define the following partial order on Dehn functions.

\begin{definition}\label{def:dehn_equiv}
For $f,g \in \F$ we define $f \preceqD g$ if there exists a constant $C>0$ so that $f(n)\le Cg(Cn)+Cn$ for all $n\in\N$.
We write $f \approxD g$ if  $f \preceqD g$ and $g \preceqD f$.
\end{definition}

Recall that each presentation of a group $G$ can give rise to a different Dehn function.
It is a standard fact
that all Dehn functions on a group $G$ are equivalent under the relation $\approxD$.  Thus we can
consider the equivalence class of these functions as a quasi-isometry invariant of the group.  In particular,   we can refer to a group as possessing a linear, quadratic or exponential Dehn function, for example.

Recall that
there are no groups with Dehn function equivalent to $n^\alpha$ for $\alpha \in (1,2)$ \cite{Bowditch,GromovHyp, Olshanskii}.

 \section{Finite presentability and Dehn functions}\label{sec:FP-dehn}

We start with the following observation.
\begin{lemma}[\cite{KKMjournal}, Lemma 8.2; \cite{cga}, Lemma 8]\label{lem:bound1}
Let $(S,\Lambda,L,\psi)$ be a  Cayley automatic structure  for $G$. Then there are constants $m,e\in \N$, depending on the Cayley automatic structure, with $m\geq 1$ so that for each $u \in L$,
\[|u| \leq m d_S(1_G, \psi(u)) + e\]
where $d_S$ denotes the word metric in $G$ with respect to the generating set $S$.
\end{lemma}

\begin{proof}
For each $x\in S$ let $\texttt{M}_x$ be a synchronous 2-tape automaton accepting
\[\{(u,v)\in L^2\mid \psi(v)=_G\psi(u)x\},\]
and let $|\texttt{M}_x|$ denote the number of states in $\texttt{M}_x$, and
$m=\max\{|\texttt{M}_x|\mid x\in S\}$. Let $u_0\in L$ be such that $\psi(u_0)=1_G$, and $e=|u_0|$.

For  $u\in L$ let
$x_1\dots x_k$ be a geodesic  for $\psi(u)$ where $x_i\in S$, so $k={d_S(1_G,\psi(u))}$. Define $u_i\in L$ by $\psi(u_i)=_Gx_1\dots x_i$. Then for $i=1,\dots, k$ we have
 \[\left||u_i|-|u_{i-1}|\right|\leq m.\]
If this difference in length was greater than $m$, the path accepted by the two-tape automaton would end with a sequence of $\$$ symbols in one coordinate of length greater than $m$.
One could then apply the pumping lemma to this path, and contradict the fact that $\psi$ is a bijection.

It then follows from the triangle inequality that
\[\begin{array}{lll}
|u|&=&\left||u_k|-|u_{k-1}|+|u_{k-1}|-\dots -|u_1|+|u_1|-|u_{0}|+|u_0|\right|\\
&\leq &\left|\left(|u_k|-|u_{k-1}|\right)\right|+\dots +\left|\left(|u_1|-|u_{0}|\right)\right|+|u_0|\\
&\leq &m k+e\end{array}\]
 which establishes the bound.
\end{proof}

The following proposition relates the Cayley distance function to fillings of loops in the Cayley graph of a Cayley automatic group.

\begin{proposition}
\label{prop:dehnbound}
Let $(S,S,L,\psi)$ be a  Cayley automatic structure  for $G$ with \distfun\ $h_{S,\psi}$.
There exist constants
$\cA,\cB,\cC,\nD, D\in\N$, depending on the Cayley automatic structure, so that the following holds.
\begin{enumerate}\item
For every $w\in S^*$ with $w=_G1_G$ and $|w| \geq \nD$,
there exist $w_i,\rho_i\in S^*$ with $w_i=_G1_G$ for  $1\leq i\leq k$  so that \[w=_{F_S}\prod_{i=1}^k\rho_iw_i\rho_i^{-1}\ \ \ \ \
 \text{and} \ \ \ \ \ |w_i|\leq 4h_{S,\psi}(
 \cA|w|+\cB)+\cC.\]
\item
If  $G$ is finitely presented, and $\delta$ is the Dehn function with respect to a fixed presentation $\langle S\mid R\rangle$, then
\[\delta(n) \leq D n^2 \delta(f(n))\] for all $n\geq n_0$,
where $f\approxF h_{S,\psi}$.
\end{enumerate}
\end{proposition}

Note that the constants $D,c$, and $\cC$ in the statement of the proposition depends only on the Cayley automatic structure and not on the Dehn function $\delta$.

\begin{proof}
Let $m=\max_{s\in S}\{|\texttt{M}_s|\}$ be the maximum number of states in any two-tape synchronous automaton accepting $R_s$ as in Definition~\ref{def:Caut} in the Cayley automatic structure for $G$ and  $u_0\in L$  the word representing the identity element of $G$  of length $e$ as in Lemma~\ref{lem:bound1}.
Without loss of generality we assume that $m$ is
	even, so that all arguments of the function $h_{S,\psi}$
  below are integers.

Choose a loop in $\Gamma(G,S)$ based at $1_G$ labeled by the path $w=s_1\dots s_n$ where $s_i\in S$ and $\pi(w)=1_G$.  For each $g_i=\pi(s_1\dots s_i)$ let $u_i\in L$ be such that $\psi(u_i)=g_i$.
For $1\leq i\leq n$,
as $d(1_G,g_i)\leq n/2$, it follows from Lemma~\ref{lem:bound1}  that \begin{equation}\label{eq:lemma}
    |u_i|_S\leq m n/2+e
\end{equation} so the distance from $\pi(u_i)$ to $g_i$ is at most
$h_{S,\psi}(mn/2+e)$.
Let $\gamma_i$ be a path from $\pi(u_i)$ to $g_i$ of length at most this bound.
We will describe how to fill ``corridors" having perimeter $u_i\gamma_is_i\gamma_{i+1}^{-1}u_{i+1}^{-1}$ with relators of bounded perimeter.
See  Figure~\ref{fig:filling1} for an example of such a corridor.

 \begin{figure}[h!]
\input{loop1.tex}
\caption{The exterior of the figure is labeled by a loop $w = s_1s_2 \cdots s_n$ with $w=_G1$.  The figure depicts a corridor whose sides are labeled by the closed path $u_i\gamma_is_i\gamma_{i+1}^{-1}u_{i+1}^{-1}$.}
\label{fig:filling1}
\end{figure}
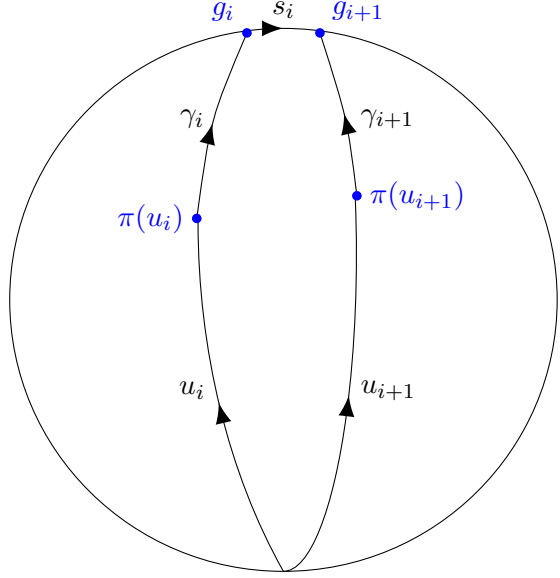

Let $u_i=a_{i,1}a_{i,2}\dots a_{i,|u_i|_S}$ and for $0\leq j\leq |u_i|_S$ define $p_{ij}=\pi(a_{i,1}a_{i,2}\dots a_{i,j})$ to be the point in $\Gamma(G,S)$ corresponding to the prefix of $u_i$ of length $j$.
If $j>|u_i|_S$ then let $p_{ij}=\pi(u_i)$.

We know that the pair $(u_i,u_{i+1})$ is accepted by $M_{s_i}$.
Consider the state  of $M_{s_i}$ which is reached upon reading the input $\{(a_{i,l},a_{i+1,l})\}_{j=1}^l \}$
where $a_{i,l},a_{i+1,l}\in S\cup\{\$\}$.
There must be a path of length at most $m$ in $M_{s_i}$ from this state to some accept state of $M_{s_i}$.
Denote the labels along this path by $\{(b_{i,j,r},b_{i+1,j,r}) \}_{r=1}^m$ where
$b_{i,j,r},b_{i+1,j,r}\in S\cup\{\$\}$ and we insert the padding symbol $\$$ in both coordinates if the path has length less than $m$.
Then if $x_{i,j}$ denotes the concatenation
$\{a_{i,j}\}_{j=1}^l\{b_{i,j,r}\}_{r=1}^m$, and $x_{i+1,j}$ denotes
$\{a_{i+1,l}\}_{l=1}^j\{b_{i+1,j,r}\}_{r=1}^m$, then $\psi(x_{i,j})s_i = \psi(x_{i+1,j})$, and both these points, as well as $\pi(x_{i,j})$ and $\pi(x_{i+1,j})$ are depicted in Figure~\ref{fig:filling2}.

 \begin{figure}[h!]
\input{loop2}
\caption{Depiction of a path in $\Gamma(G,S)$ between $p_{i,j} = \pi(a_{i,1}a_{i,2}\dots a_{i,j})$ and $p_{i+1,j} = \pi(a_{i+1,1}a_{i+1,2}\dots a_{i+1,j})$, where $u_i=a_{i,1}a_{i,2}\dots a_{i,|u_i|_S}$ is such that $\psi(u_i) = g_i$.}\label{fig:filling2}
\end{figure}
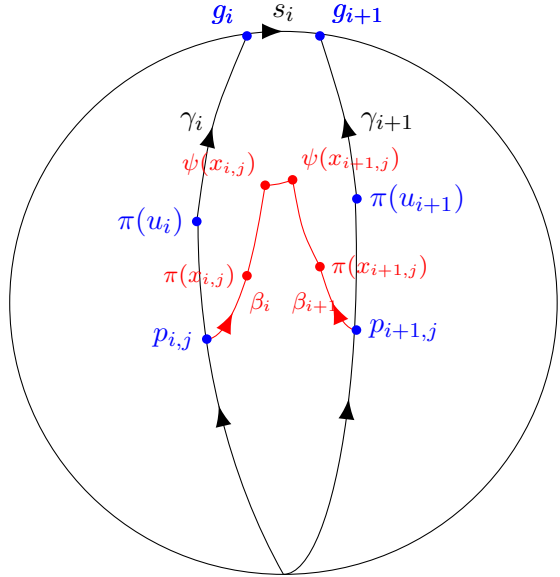

Thus there is a path in $\Gamma(G,S)$ from $p_{ij}$ to $p_{i+1,j}$ of length at most $2m+2h_{S,\psi}(mn/2+e+m)+1$
consisting of the following segments, as shown in Figure~\ref{fig:filling2}:
\begin{itemize}[itemsep=5pt]
\item $\beta_i=\{b_{i,j,r}\}_{i=1}^r$ from $p_{i,j}$ to $\pi(x_{i,j})$ of length at most $m$,
\item a path from $\pi(x_{i,j})$ to $\psi(x_{i,j})$ of length at most $h_{S,\psi}(mn/2+e+m)$,
\item an edge labeled $s_i$ from $\psi(x_{i,j})$ to $\psi(x_{i+1,j})$,
\item a path from $\psi(x_{i,j})s_i = \psi(x_{i+1,j})$ to $\pi(x_{i+1,j})$ of length at most $h_{S,\psi}(mn/2+e+m)$,
\item $\beta_{i+1}=\{b_{i+1,j,r}\}_{i=1}^r$ from $\pi(x_{i+1,j})$ to $p_{i+1,j}$ of length at most $m$.
\end{itemize}

In Figure~\ref{fig:filling3} the paths between $p_{i,j}$ and $p_{i+i,j}$ for all $1 \leq j < \max\{|u_i|_S,|u_{i+1}|_S\}$ are  depicted for one corridor.
Between $p_{|u_i|}$ and $p_{|u_{i+1}|}$ we use that existing path
$\gamma_i s_i \gamma_{i+1}^{-1}$.

These corridors create two types of cells.  The first
type are created from two of these paths and their connecting edges, for some $j$ and $j+1 < \max\{|u_i|_S,|u_{i+1}|_S\}$.
This creates a cell with perimeter at most
\[2(2m+2h_{S,\psi}(mn/2+e+m)+1) + 2 = 4m+4h_{S,\psi}(mn/2+e+m)+4,\] where the additional $+2$ accounts for the single edges $a_{i,j+1}$   between $p_{i,j}$ and $p_{i,j+1}$, and   $a_{i+1,j+1}$  between $p_{i+1,j}$ and $p_{+1,j+1}$, which lie on the paths $u_i$ and $u_{i+1}$, respectively, and are not part of the paths previously constructed.

\begin{figure}[h!]
\input{loop3}
\caption{Filling the corridors created by the paths $u_i\gamma_i$ with cells of bounded perimeter.}\label{fig:filling3}
\end{figure}
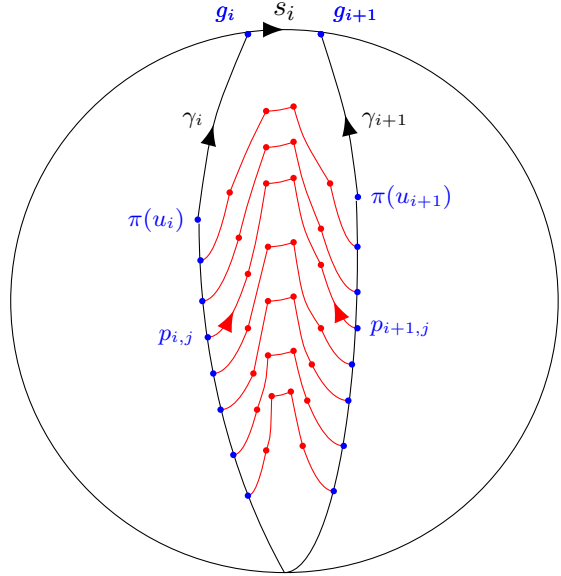

The second type is the ``top" cell created by the path from $p_{i,|u_i|-1}$ to $p_{i+1,|u_{i+1}|-1}$ together with the path
$a_{i,|u_i|-1}\gamma_is_i\gamma_{i+1}^{-1}a_{i+1,|u_{i+1}|-1}$.
This cell has perimeter at most \[\begin{array}{lll}
&
(2m+2h_{S,\psi}(mn/2+e+m)+1)+2+(  2h_{S,\psi}(mn/2+e)+1)\\
=& 2m+2h_{S,\psi}(mn/2+e+m)+2h_{S,\psi}(mn/2+e)+4\\
\leq & 4m +4h_{S,\psi}(mn/2+e+m)+4\end{array}\]
where the terms in the first line come, respectively, from
\begin{itemize}[itemsep=5pt]
    \item the path from $p_{i,|u_i|-1}$ to $p_{i+1,|u_{i+1}|-1}$,
    \item the two edges labeled $a_{i,|u_i|-1}$ and $a_{i+1,|u_{i+1}-1}$, and
    \item the path $\gamma_is_i\gamma_{i+1}^{-1}$.
\end{itemize}
To obtain the inequality, note that $h_{S,\psi}\in \F$ and $m\geq 1$, so $2m+4\leq 4m+4$ and \[2h_{S,\psi}(mn/2+e+m)+2h_{S,\psi}(mn/2+e)\leq 4h_{S,\psi}(mn/2+e+m).\]
Setting $c=m/2, d=e+m$ and $\cC=4m+4$ proves the first claim in the proposition.

To prove the second claim in the proposition, we count the total number of cells required to subdivide the initial loop into cells of bounded perimeter.
This will yield the inequality involving the Dehn function $\delta$.

It follows from Lemma~\ref{lem:bound1} that $|u_i|\leq mn/2+e$
for all $1\leq i\leq n$.
Each corridor is filled by at most $(mn/2+e)$ cells, each of perimeter at most $4h_{S,\psi}(cn+d)+\cC$, where $c,d$ and $\cC$ depend on $m$ and $e$.
For  a fixed finite  presentation $\langle S\mid R\rangle$ for $G$ with Dehn function $\delta$, each cell constructed above can be filled by at most  $\delta(4h_{S,\psi}(cn+d)+\cC)$ cells with perimeter labeled by a relator from the set $R$.

With $n$ corridors, there are $n\cdot (mn/2+e)=n(cn+e)$ such cells to fill.
Thus an upper bound on the number of relators required to fill $w$ is
\begin{dmath*}
n(cn+e)\cdot \delta\left(4h_{S,\psi}(cn+d)+\cC\right)
=n^2(c+\frac{e}{n})\cdot \delta\left(4h_{S,\psi}(cn+d)+\cC\right)\\
\leq n^2(c+e)\cdot \delta\left(4h_{S,\psi}(cn+d)+\cC\right).
\end{dmath*}
Setting $D=c+e$
and noting that it follows from Lemma~\ref{lemma:lose_the_constant} that $4h_{s,\psi}(cn+d)+\cC\approxF h_{s,\psi}(n)$
proves the second claim of the proposition.
\end{proof}

\section{Separating finitely presented Cayley automatic groups from automatic groups}\label{sec:FP}

In this section we prove Theorem~\ref{thmA:fp}.
First
we  introduce the following  notion.

\begin{definition}
	Let  $f,g\in \F$.
	We say that
	$f \ll g $
	if there exists an unbounded
	function $t \in \F$ such that $ft \preceqF g$. 	
\end{definition}

\begin{example}
If $g(n)=n^c$ with $c>2$ and $f(n)=n^2$ then $f\ll g$.
Take $t(n)=n^{c-2}$.  Then $t \in \F$ is an unbounded function and \[f(n)t(n) = n^c \preceqF g(n).\]
\end{example}

Next we define the following.

\begin{definition}\label{defn:strong-super}
     	A function $f\in \F$ is
{\em  super-quadratic} if for all constants
 	$M > 0$ we have $f(n) \leqslant M n^2$
 	for at most finitely many $n \in \mathbb{N}$.
    A non-zero function is $f\in \F$ is {\em  strongly-super-polynomial} if $n^2f\ll f$.
\end{definition}

\begin{example}
The functions $n^2\ln n$ and  $n^c$ for $c>2$ are super-quadratic;  the functions $e^n$ and  $n^{\ln n}$ are strongly-super-polynomial.
\end{example}

\Olshan\ introduces the notion of a function being {\em almost quadratic} in \cite{Ol-almost}; our definition of a super-quadratic function is the same as being not almost quadratic.
However, our notion of a strongly-super-polynomial is
stronger than the more standard definition of a super-polynomial function given, for example, in \cite{GrigPak}:

\begin{definition}\label{defn:superpoly}
   A function $f\colon \N\to \R$
   is
   {\em super-polynomial} if $$ \lim_{n \to \infty} \frac{\ln f(n)}{\ln n} = \infty.$$
\end{definition}

In Lemma~\ref{lem:NotStrongSuperP}
we give an example of a function in $\F$ which satisfies the above limit but is not strongly-super-polynomial. However  Proposition~\ref{prop:strong_implies_super} justifies our use of ``strongly" since it shows that every strongly-super-polynomial function is super-polynomial.

Proposition~\ref{prop:strongSuper} shows that $f$ is strongly-super-polynomial if and only if $n^cf\ll f$ for any $c>0$, that is, there is nothing special about the choice of the exponent $2$ in Definition~\ref{defn:strong-super}.

\begin{lemma}\label{lem:equiv-super-quad}
	A function $f \in \F$
    is super-quadratic if
    and only if
    $n^2 \ll f$.
\end{lemma}
\begin{proof}
    Assume first that $n^2 \ll f$.
    Since $n^2 \ll f$, there exist an unbounded
    function $t \in \mathcal{F}$ and integer constants
    $K,N>0$ and $M \geqslant 1 $ such that
    $n^2 t(n) \leqslant  K f (M n)$ for all
    $n \geqslant N$.
    Assume that for some $M'\geqslant 1$ there exist
    infinitely many $n_i \in \mathbb{N}$, $i\geq 1$ with $1\leq n_i < n_{i+1}$ for which
    $f(n_i) \leqslant M' n_i^2$.
    Let $n_i = k_i M + r_i$, where $k_i$ is an integer
    and $0 \leqslant r_i < M$. Then we have:
    \begin{equation*}
    \begin{split}
       \frac1{K}  k_i ^2  t(k_i)\leqslant f (M k_i)
        \leqslant f (n_i)
        \leqslant M' n_i^2 = M' M^2 k_i^2 +
        2 M' M r_i k_i + M' r_i ^2 \leqslant \\
        M' M^2 k_i ^2 + 2 M' M^2 k_i + M' M^2
        \leqslant M' M^2 (k_i +1)^2
    \end{split}
    \end{equation*}
    for all $k_i \geqslant N$.
    Therefore, $t(k_i) \leqslant KM'M^2
    \frac{(k_i + 1)^2}{k_i^2} \leqslant 2 KM'M^2$ for all
    $k_i \geqslant \max\{N, 3\}$, where the $3$ follows from
    the simple observation that if
    $k \geqslant 3$ then $\frac{(k+1)^2}{k^2} \leqslant 2$.
    This contradicts the fact that $t$ is an unbounded function.

    Now assume that $f$ is super-quadratic. Then
    for each integer $i \geqslant 1$ the set
    \[ \{m \, |\,\forall n
    \left[n \geqslant m \implies f(n) \geqslant i n^2
    \right]\}\] is non-empty. Let $m_i=\min  \{m \, |\,\forall n
    \left[n \geqslant m \implies f(n) \geqslant i n^2
    \right]\}$.
    We define a function $t(n)$ as follows:
     for $0 \leqslant n < m_1$, let $t(n)=0$, and for
 $m_i \leq n < m_{i+1}$ with $i\geq 1$,  let $t(n) = i$.
    By construction, $t(n)$ is a nondecreasing and unbounded function.
    As $n^2 t(n) \leqslant f(n)$, it follows
    that $n^2 \ll f$.
 \end{proof}

 In \cite{WeirdQuadDehn}  \Olshan\ gives an example of a finitely presented group  which has Dehn function bounded above  by $c_1n^2$ for infinitely many values of $n$,  bounded below by $c_2n^2\log' n\ \log'\log' n$ for infinitely many values of $n$, where $\log'(n)=\max\{\log_2 n, 1\}$, and bounded between $c_3n^2$ and $c_4n^2\log'n\ \log'\log' n$ for all $n\in \N$. Since this Dehn function is not super-quadratic, it follows that  Theorem~\ref{thmA:fp} below does not apply to this example.

We now  prove Theorem~\ref{thmA:fp}.

\ThmA*
\begin{proof}
Fix a presentation for $G$ and let $\delta$ be the Dehn function arising from this presentation. Fix a Cayley automatic structure $(S,S,L,\psi)$ for $G$.
If  $n^2 \ll \delta$,  there is an unbounded function $t(n) \in   \F$ and positive constants $K,M,N_0 \in \N$ so that $n^2t(n) \leq K\delta(Mn)$ for all $n \geq N_0$.

From Proposition~\ref{prop:dehnbound} we know that there are constants $N_1, D\geq 1$ and a function $f\in \F$ so that  for $n \geq N_1$, we have  $\delta(n) \leq Dn^2\delta(f(n))$ where $f\approxF h_{S,\psi}(n)$.

Combining these equations we have  that for all $n\geq \max\{N_0,N_1\}$
\[n^2t(n)\leq K\delta(Mn)\leq KDM^2n^2\delta\left(f(Mn)\right)\]
and dividing both sides by $KDM^2n^2$ we obtain
\begin{equation}\label{egnINEQ}
\frac{t(n)}{KDM^2}\leq  \delta\left(f(Mn)\right).\end{equation}

Define
$$\phi(n) = \min\left\{m\, \middle|\,\frac{t(n)}{KDM^2} \leq \delta(m)\right\}.$$

It is immediate from Equation~(\ref{egnINEQ}) that in the definition of $\phi(n)$ we have $\phi(n) \leq f(Mn)$ for all $n\geq \max\{N_0,N_1\}$, and hence $\phi \preceqF f$.  Since $t \in   \F$ is unbounded, it follows that $\phi \in \F$ and $\phi$ is unbounded.

Now assume that the inequality $n^2 \delta \ll \delta$
is satisfied.
Therefore, there exist integer constants $K,M,N_0 >0$ and an unbounded function $t \in \mathcal{F}$ such that
    \begin{equation}
    \label{ineq_n^2_delta_ll_delta}
      n^2 \delta (n) t(n) \leqslant K \delta (M n)
    \end{equation}
    for all $n \geqslant N_0$.
    It follows from statement (2) of Proposition 3.2. that there exists
    a function $f \approx_1 h_{S,\psi}$ and integer constants
    $N_1 \geqslant 0$ and $D>0$ for which
    the inequality
    \begin{equation*}
     \delta(n) \leqslant D n^2 \delta (f(n))
    \end{equation*}
    holds for all $n \geqslant N_1$.
    This implies that
    $\delta (Mn) \leqslant DM^2 n^2 \delta (f (Mn))$ for all
    $n \geqslant N_1$.
    Combining this with the inequality in \eqref{ineq_n^2_delta_ll_delta} we obtain
    that
    $$n^2 \delta (n) t(n) \leqslant K \delta (M n)
     \leqslant D K M^2 n^2 \delta (f(Mn))
    $$
    for all $n \geqslant  \max \{N_0, N_1\}$.
    Therefore,
    \begin{equation*}
       \delta (n) \tau (n) \leqslant  \delta (f(Mn))
    \end{equation*}
    for all $n \geqslant  \max \{N_0, N_1\}$, where
    $\tau (n) =  \frac{t(n)}{D K M^2}$.
    Let \[m_0 = \min\{n\in \dom(\tau)\subseteq  \N \, | \, \tau(n) \geqslant 2 \};\]
    such $m_0$ exists because $\tau(n)$ is unbounded.
    Therefore,
    \begin{equation*}
    \label{semifinal_ineq}
       2 \delta (n) \leqslant \delta (f(Mn)))
    \end{equation*}
    for all $n \geqslant \max \{N_0, N_1, m_0\}$.
    Let  $d_0 = \min\{n \, | \, \delta (n)
    	\geqslant 1\}$.
    	
    	If $f(Mn)<n$ for some
    $n \geqslant \max\{N_0, N_1, m_0, d_0\}$
    then $2\delta(n)\leq \delta(f(Mn))\leq \delta(n)$, which is a contradiction.
    Thus
for all $n\geq \max \{N_0, N_1, m_0, d_0\}$ we must have that
    \begin{equation*}
    \label{final_ineqAA}
        n \leqslant f(Mn).
    \end{equation*}
   From this we obtain that
   $\ii \preceqF f$. As $f \preceqF h_{S,\psi}$ it follows that
   $\ii\preceqF h_{S,\psi}$, and we conclude that $G$ is $\ii$-separated.
\end{proof}

\section{Dense groups}\label{sec:dense}

We introduce a property of some infinitely presented groups which will allow us to obtain sharper lower bounds on the Cayley distance function of such a Cayley automatic group.
This property will be shown to be independent of generating set and the prototypical examples of groups with this property are restricted wreath products.

Recall that  $F_X$ denotes the free group generated by a set $X$.

\begin{definition}[Densely generated]
\label{def:dense}
Let $G$ be a group with finite  generating set  $X$.
We say that $G$ is   {\em densely generated} by $X$
if  there exist  constants $E,F, N_0\in \N$,
 $1\leq E<F$
such that  for all $n\geq N_0$ there is a word $w_n\in (X\cup X^{-1})^*$ which has the following properties:
 \begin{itemize}[itemsep=5pt]
 \item $w_n=_G1_G$,
 \item $En\leq |w_n|\leq Fn$, and
 \item for any collection of words $u_i,\rho_i\in (X\cup X^{-1})^*$,   $1 \leq j \leq k$ with $u_i=_G1_G$ and
 \[ w_n=_{F_X}\prod_{i=1}^k\rho_iu_i\rho_i^{-1}, \]
 we have $|u_j|>n$ for some  $1 \leq j \leq k$.
 \end{itemize} \end{definition}
In other words, for every interval $[En,Fn]$ there is a loop $w_n$ whose length lies in that interval which  cannot be {filled} by loops all having length at most $n$.
It follows that if $G$ is densely generated by $X$ then every presentation for $G$ over $X$ is infinite.

The following lemma shows that being densely generated is independent of the choice of finite generating set.

\begin{lemma}\label{lem:denseGsetInvariant} If $G$ is {densely generated} by $X$ and $Y$ is another finite generating set for $G$ then $G$ is {densely generated} by $Y$.\end{lemma}

\begin{proof}
Let $| \cdot |_X$ denote the length of a word in $(X \cup X^{-1})^*$ and $| \cdot |_Y$ denote the length of a word in $(Y \cup Y^{-1})^*$.

For each $x\in X$ choose a nonempty word $v_x\in (Y\cup Y^{-1})^*$ with $x=_Gv_x$. Let $M_1=\max_{x\in X}\{|v_x|_Y\}$, and $\tau\colon (X\cup X^{-1 })^*\to (Y\cup Y^{-1})^*$ be  the monoid homomorphism defined by $\tau(x)=v_x$.
For each $y\in Y$ choose a nonempty word $q_y\in (X\cup X^{-1})^*$ with $y=_Gq_y$. Let $M_2=\max_{y\in Y}\{|q_y|_X\}$, and $\kappa\colon (Y\cup Y^{-1 })^*\to (X\cup X^{-1})^*$ the monoid homomorphism defined by $\kappa(y)=q_y$.

As $G$ is densely generated by  $X$, there exist fixed constants $E,F,N_0 \in \N$ as in Definition~\ref{def:dense}.
Suppose $G$ is not densely generated by $Y$. Then for all constants $E',F', N_0' \in \N$ there exist some $s\geq N_0'$
so that all words equal to $1_G$ of length between $E's$ and $F's$ can be filled by cells of perimeter at most $s$.

Choose $E'=EM_2$ and $F'=M_1M_2F$, and $N_0'=\max\{N_0,M_1M_2+1\}$.
Let $s_0\geq N_0'$ be chosen so that with respect to these constants, all words equal to $1_G$ of length between $E's_0$ and $F's_0$ can be filled by cells of perimeter at most $s_0$.

As $G$ is densely generated by $X$, choose $n=M_2s_0$. There must be a word $w_n\in (X\cup X^{-1})^*$ so that $w_n =_G 1_G$ and whose length satisfies
\[E(M_2s_0)\leq |w_n|\leq F(M_2s_0)\]
which cannot be filled by cells all of perimeter at most $n$.
Then $\tau(w_n)$ labels a path in $\Gamma(G,Y)$ so that
\[EM_2s_0\leq |\tau(w_n)|\leq M_1FM_2s_0.\]
Note that the map $\tau$ does not decrease length, as $\tau(w_n)$ is obtained by substitution, with no free reduction.

Since $E's_0\leq |\tau(w_n)|\leq F's_0$, by our choice of $s_0$ we can fill this word
by cells of perimeter at most $s_0$. Now consider this van Kampen diagram as a subgraph of $\Gamma(G,Y)$.
Map the entire subgraph, edge-by-edge, into $\Gamma(G,X)$ by applying the map $\kappa$; the boundary of the new subgraph consists of paths of the form $\kappa(\tau(x_i))$, where $w = x_1x_2, \cdots x_n$.
These paths form the boundary of the subgraph in $\Gamma(G,X)$ connecting the original vertices on the path labeled by $w_n$, and have length at most $M_1M_2$.  We have thus created cells of the form $\kappa(\tau(x_i))x_i^{-1}$ of perimeter at most $1+M_1M_2$.
These boundary cells, together with the copy of the van Kampen diagram, provide a filling of $w_n$.

In summary, the filling we have created has cells of two types:
\begin{itemize}
    \item the boundary cells, of perimeter $1+M_1M_2$, and
    \item images of the cells in the van Kampen diagram in $\Gamma(G,Y)$, which had perimeter at most $s_0$; after applying the homomorphism $\kappa$, the image of such a cell has perimeter at most $M_2s_0$.
\end{itemize}

Note that we chose $N_0'=\max\{M_2M_1+1,N_0\}$ so all of these cells have perimeter at most $M_2s_0=n$.
This contradicts the existence  of $w_n$. 
Thus $G$ is also densely generated by $Y$.
\end{proof}

A group is called {\em dense} if it is densely generated by some, hence any, finite generating set.
This definition is inspired by  Baumslag's paper \cite{baumslag61} about wreath products $G\wr H$.
We prove in Proposition~\ref{prop:BaumslagDense} that if $H$ is infinite then $G\wr H$ is  dense.

\begin{proposition}\label{prop:BaumslagDense}
  Let $G$ and $H$ be finitely generated groups.	  If $G$ is nontrivial and $H$ is infinite, then
     $G \wr H$ is dense.
\end{proposition}
\begin{proof}
For $x,y\in G$ let $x^y=y^{-1}xy$, and $[x,y]=x^{-1}y^{-1}xy$. 
Let $G = \langle X \, |\,  P \rangle$ and
$H = \langle Y \, | \, Q \rangle$ be presentations of
the groups $G$ and $H$, where
$X \subseteq G$ and $Y \subseteq H$ are finite generating sets.
For each $h \in H$ choose a geodesic word $u_h \in (Y \cup Y^{-1})^*$ with $\pi(u_h) =_G h$, and let
$U = \{ u_h | h \in H\}$.
Then the wreath product $G \wr  H$ has presentation
\[G \wr  H = \langle  X \cup Y \, | \,
P \cup Q \cup \{[a_1^u,a_2^v]\,|\,
a_1,a_2 \in X, u,v \in U, u\neq v\} \rangle.\]

Let $B_{H,Y} (n)$ denote the ball
of radius $n$ in the group $H$ with respect to the
generating set $Y$.
For a given positive integer $m$, define the set of relators
\[ R_m =P \cup Q \cup  \{[a_1^u,a_2^v]
\mid u,v \in U,u\not=v,\pi(u),\pi(v) \in B_{H,Y}(m)\}.\]

For any set $S \subseteq H$, define the relation
$T_S = \{ (s_1 h,s_2 h) \mid s_1,s_2
\in S, h \in H\}$.
Now we mimic Baumslag's argument for
proving the non-finite presentability of
wreath products, presented in Lemma~3 of  \cite{baumslag61}, see also  \cite{berstein2015}.
Baumslag constructs a group $\BaumGroup$
generated by $G$ and $H$ with the following properties: 
\begin{itemize}[itemsep=5pt]
    \item $G^{h_1}  \cap G^{h_2} =\{1_\BaumGroup\}$ for all $h_1,h_2 \in H$ with $h_1 \not= h_2$, and
    \item $\left[ G^{h_1}, G^{h_2}\right]= \{1_\BaumGroup\}$ if and only if $(h_1,h_2) \in T_S$.

\end{itemize}
Note that instead of requiring all conjugacy classes to commute in $\BaumGroup$, we only require this when the conjugating elements form a pair in the relation $T_S$.

Choosing $S = B_{H,Y}(n)$ for any fixed $n$, it follows that in $\BaumGroup $ we have $[G,G^h]\not=1_\BaumGroup$
for any $h$ for which
$(1_H,h) \not\in T_S$.
In particular, this holds for any
$h \in B_{H,Y} (2n+1) \setminus B_{H,Y} (2n)$.
Therefore there is a relation $[a_1,a_2 ^h]$ in
$R_{2n+1} \setminus R_{2n} $ which cannot be obtained
as a product of conjugates of the relations from
$R_n$.

Now observe that every loop $w \in (X \cup X^{-1} \cup Y \cup Y^{-1})^*$
of length $|w| \leqslant n$
in the wreath product $G \wr H$
can be represented as a product of
conjugates of relations from $R_n$.
Therefore, the loop of length $8n + 8$,
given by the relation $[a_1,a_2 ^h]$,
cannot be decomposed into
smaller loops of length less or equal than $n$.
Thus, $G \wr H$ is dense.
\end{proof}

\section{Separating non-finitely presented Cayley automatic groups from automatic groups}\label{sec:thmB}

The proof of Theorem~\ref{thmB:nfp} relies 
on the following proposition.

\begin{proposition}
   \label{prop:nonFP}

Let $G$ be a non-finitely presented group with  finite generating set $S$.  Then there exists a non-decreasing step function $\phiGS \in\F$, depending on $G$ and $S$,  and  an infinite sequence of integers $\{n_i\}$ such that $\phiGS(n_i)=n_i$ and for any Cayley automatic structure $(S,S,L,\psi)$ on $G$,
\begin{enumerate}
    \item
    $\phiGS \preceqF h_{S,\psi}$, and
    \item if $G$ is dense then
    $\mathfrak i\preceqF h_{S,\psi}$.
\end{enumerate}
\end{proposition}

\begin{proof}
Since $G$ is not finitely presented, there exists an infinite sequence of words  $w_i\in S^*$ so that
\begin{itemize}[itemsep=5pt]
    \item $w_i =_G 1_G$,
    \item if $w_i=\prod_{j=1}^k \rho_ju_j\rho_j^{-1}$ for some $k \in \N$ and $u_j, \rho_j \in S^*$  for $1\leq j\leq k$, then $|u_j|\geq |w_i|$ for at least one value of $j$, and
    \item $|w_i| = l_i$, and $l_i < l_{i+1}$.
\end{itemize}
Define $\phiGS \in   \F$ by $\phiGS(n) = l_i$ for $l_i \leq  n <  l_{i+1}$.

Let $c,d,\cC$, and $n_0$ be the constants from  Proposition~\ref{prop:dehnbound}. Then for any $i\in\N$ with $l_i\geq n_0$ we can
decompose $w_i$ into  loops $u_{i,j}$ using the algorithm described in Proposition~\ref{prop:dehnbound}, and illustrated in Figures~\ref{fig:filling1}, \ref{fig:filling2} and \ref{fig:filling3}, so that $|u_{i,j}| \leq 4h_{S,\psi}(cl_i+d) + \cC$.
Our choice of $w_i$ ensures that for some $j$ we have  $\phiGS(l_i) = l_i = |w_i| \leq |u_{i,j}|$.

Suppose $l_i \leq n < l_{i+1}$.
It follows that for this choice of $j$,
\[ \phiGS(n) = l_{i}\leq |u_{i,j}| \leq 4h_{S,\psi}(c l_{i} + d) + \cC \leq 4h_{S,\psi}(cn+d)+\cC. \]
It then follows from Lemma~\ref{lemma:lose_the_constant} that
 $\phiGS(n) \preceqF h_{S,\psi}(n)$.

Now suppose that $G$ is densely generated by $X$, so there exist constants $E,F$ and $N_0$ so that for all $n \geq N_0$ there exists a loop $w_n=_G 1_G$ so that
\bi
\item $En \leq |w_n| \leq Fn$, and
\item $w_n$ cannot be subdivided into loops all of whose lengths are bounded above by $n$.  That is, if we write $w_n = \Pi_{i=1}^k \rho_i u_i \rho_i^{-1}$ where each $u_i =_G 1_G$ then for some $i$ we have $|u_i| \geq n$.
\ei
Again it follows from Proposition~\ref{prop:dehnbound} that there are constants $c,d,\cC$ and $n_0$ so that for $n \geq \max(n_0,N_0)$, each $u_j$ in the above decomposition of $w_n$ we have
$|u_j| \leq 4h_{S,\psi}(cFn+d)+\cC.$
Since $G$ is dense, it follows that for some $j$ we have
$$n \leq |u_j|\leq 4h_{S,\psi}(cFn+d)+\cC.$$
As this is true for every $n \in \N$ with $n \geq \max(n_0,N_0)$,
it follows from Lemma~\ref{lemma:lose_the_constant} that $\ii \preceqF h_{S,\psi}$.
\end{proof}

We now prove Theorem~\ref{thmB:nfp}.

\ThmB*
\begin{proof}
Suppose $(Y,Y, L, \psi)$ is a Cayley automatic structure for $G$ with respect to some arbitrary finite generating set $Y$.

If 
$G$ is dense.
 it follows from part (2) of Proposition~\ref{prop:nonFP} that  $\ii\preceqF  h_{Y,\psi}$.

Next, choose a finite generating set  $S$ for $G$. 
It follows from part (1) of Proposition~\ref{prop:nonFP} that the function $\phiGS$ is such that $\phiGS\preceqF h_{S,\psi'}\approxF h_{Y,\psi}$. Setting   $\phi=\phiGS$ gives a  step function which suffices  to prove the theorem.
\end{proof}

\section{Conclusion}

Theorems~\ref{thmA:fp} and~\ref{thmB:nfp} together imply that the only possible candidates for non-automatic Cayley automatic
groups where the geometry comes close to resembling that of an automatic group (in the coarse sense considered here) 
are groups with
quadratic or almost quadratic Dehn function. The class of groups with quadratic Dehn function
is a wild and interesting collection (Gersten referred to the class as a “zoo” in \cite{Gersten}),
including the following 
 groups.
\begin{enumerate}
\item The higher Heisenberg groups $H_{2k+1}$ for $k>2$ \cite{MR1616147,MR1253544,MR1698761}.
\item The higher rank lamplighter groups, or Diestel-Leader groups \cite{Taback18}.
\item Stallings’ group and its generalizations \cite{CarterForester,DisonElderRileyYoung}.  These examples are not of type $FP_{\infty}$, hence not automatic; it is not known whether they admit Cayley automatic structures.
\item  Thompson’s group $F$, which is not known to be automatic or Cayley automatic though is 1-counter-graph automatic \cite{cga}.
\item  An example of \Olshan\ and Sapir \cite{OSapir-CP} which has quadratic Dehn
function and unsolvable conjugacy problem. 
\end{enumerate}

\Olshan\  \cite{WeirdQuadDehn,Ol-almost} also gives an example of a group whose Dehn function is
almost quadratic, and so Theorem~\ref{thmA:fp} does not apply to this group. It is not known whether this 
example or the unsolvable conjugacy example  of  \Olshan\  and Sapir are Cayley automatic. Note that our definition of
a super-quadratic function is equivalent to saying the function is not almost quadratic, as shown in Lemma~\ref{lem:equiv-super-quad}.

Progress towards proving Conjecture~\ref{conj1}  takes two forms.
First, one must improve the exhibited bounding functions given for groups with super-quadratic Dehn function and for non-dense non-finitely presented groups to show that these groups are $\ii$-separated.
Second, one must prove that non-automatic Cayley automatic groups with quadratic and almost quadratic Dehn function are $\ii$-separated. We have some optimism for progress on the first part, and find the second more difficult.

\begin{appendix}
\section{Further remarks  on strongly-super-polynomial functions}\label{appendix:super-strong}

In this appendix we give more details on 
strongly-super-polynomial and super-polynomial functions.

 \begin{proposition}
 \label{prop:strongSuper}	
    Let $c, d \in \mathbb{R}$
    	 such that $0< c < d$. Then for a function $f \in \mathcal{F}$, we have
    $n^{c} f \ll f$ if and only if $n^{d} f \ll f$.
 \end{proposition}
 \begin{proof}
    Assume first that $n^{d} f \ll f$. Then there exists
    an unbounded function $t \in \mathcal{F}$ and integer constants $K,M>0$ and $N \geqslant 0$
    such that
    $n^{c} f(n) t(n) \leq K f (Mn)$ for all
    $n \geq N$.
    Let $\tau  (n)= n^{d-c} t(n)$. Then $\tau \in \mathcal{F}$ and as $t$ is unbounded, so is $\tau$. Writing
    $n^{c} f(n) \tau(n)\leq K f (Mn)$, it follows that
    $n^{c} f \ll f$.

    Now assume that $n^{c} f \ll f$. Then there exists
    an unbounded function $t \in \mathcal{F}$ and
    integer constants $K,M>0$ and $N \geqslant 0$
    such that
    $n^{c} f(n) t(n) \leqslant K f (Mn)$ for all
    $n \geqslant N$. Therefore, the inequality
    \begin{equation*}
        n^{c} t (n) \leqslant K \frac{f(Mn)}{f(n)}.
    \end{equation*}
    holds for all $n \geqslant N$. This implies that
    the inequality
    \begin{equation}
    \label{main_ineq}
       (M^k n)^{c} t(M^k n) \leqslant K \frac{f(M^{k+1}n)}
       {f(M^{k}n)}
    \end{equation}
    holds for all integers $k \geqslant 0$ and $n \geqslant N$.
    Let $k_0 = \lfloor \frac{d}{c}\rfloor$; then we have
    $d \leqslant (k_0+1) c$.
    Allowing $k$ to take all values between $1$ and $k_0$ in \eqref{main_ineq} and multiplying the resulting inequalities together yields
    \begin{dmath*}
    \label{main_ineq2}
        n^{c} t(n) (M n)^{c} t(M n) \dots
       (M^{k_0} n)^{c} t(M^{k_0} n)
       \nonumber \\\leq
       K^{k_0+1} \frac{f(Mn)}{f(n)} \frac{f(M^2 n)}{f(M n)}
       \dots \frac{f(M^{k_0+1}n)}{f(M^{k_0}n)}
       =
       K^{k_0 +1}  \frac{f(M^{k_0 +1}n)}{f(n)}.
    \end{dmath*}

   It follows that
    \begin{equation}
    \label{main_ineq3}
       M' n^{c (k_0+1)}  \tau (n) \leqslant K ^{k_0 +1}
       \frac{f(M^{k_0+1}n)}{f(n)}
    \end{equation}
 holds for all $n \geqslant N$, where
 $\tau (n) = t(n) t(Mn)\dots t(M^{k_0}n)$ and
 $M' =$  $M^{c}$  $M^{2c} \dots M^{k_0 c}$.
 Since $M' \geqslant 1$, the inequality in \eqref{main_ineq3}
 implies that
 \begin{equation}
 \label{final_ineq}
    n^{c (k_0 +1)} f(n) \tau (n) \leqslant
    K^{k_0+1} f(M^{k_0 + 1}n)
 \end{equation}
 for all $n \geqslant N$. By construction, $\tau(n)$ is both an element of ${\mathcal F}$ and an
 an unbounded function.
 Therefore, $n^{c (k_0 + 1)} f \ll f$.
 As $d \leq (k_0+1)c$ it follows from the initial argument in the proof that $n^{d} f \ll f$.
\end{proof}

The following lemma presents an example of  a function which is super-polynomial
but not  strongly-super-polynomial.

\begin{lemma} \label{lem:NotStrongSuperP}
For given real number $\alpha > 1$, let $f_{\alpha}\colon \N\to\R $ be the function  defined by $f_{\alpha}(n)= \alpha^{(\ln n)^{1.5} }$. Then \begin{enumerate}
\item $f\in \F$,
\item $f_\alpha$ is super-polynomial, and
\item $f_\alpha$ is not strongly-super-polynomial.  	
\end{enumerate}
\end{lemma}	

\begin{proof}
It is clear that $ f_{\alpha}\in \F$ since $\alpha>1$.

We have \[ \lim_{n \to \infty} \frac{\ln (\alpha^{(\ln n)^{1.5}})}{\ln n}
=\lim_{n \to \infty} \frac{(\ln n)^{1.5}\ln (\alpha)}{\ln n}
=\ln (\alpha)\lim_{n \to \infty} (\ln n)^{0.5}=\infty
\]
so $f_{\alpha}$ is
   super-polynomial. Note that $\alpha>1$ so $\ln \alpha>0$.

   Suppose (for contradiction) that
$f_{\alpha}$ is strongly-super-polynomial, so  there is an unbounded
   function $t \in \mathcal{F}$ and positive integer constants $K,M$ and $N$
   such that $n^2 f_{\alpha}(n) t(n) \leq K f_{\alpha}(Mn)$
   for all $n \geq N$. This means

    \begin{equation*}
    \label{ex_ineq1}
 t(n)  \leq
      \frac{ K\alpha^{(\ln(Mn))^{1.5}-(\ln(n))^{1.5}}}{n^2 }
   \end{equation*}

   Taking the logarithm of this inequality we obtain:
   \begin{dmath*}
    \label{ex_ineq1}
    \ln(t(n))\leq \ln K +\left[(\ln(Mn))^{1.5}-(\ln(n))^{1.5}\right]\ln \alpha-2\ln n
   \end{dmath*}  which we can write
   as
   \begin{dmath*}
    \label{ex_ineq3}
        \ln(t(n))\leq \ln K +\left[(\ln(M)+\ln(n))(\ln(Mn))^{0.5}-(\ln (n))(\ln(n))^{0.5}\right]\ln \alpha-2\ln n
   \end{dmath*}
      which becomes

      \begin{dmath*}
    \label{ex_ineq4}
      \ln( t(n))
      \leqslant \ln K + (\ln( M) )(\ln (Mn))^{0.5}  \ln \alpha +
      (\ln( n) )\left[\left(\ln (Mn))^{0.5}
       - (\ln n)^{0.5} \right) \ln( \alpha)-2 \right]
   \end{dmath*}
         which becomes
      \begin{dmath*}
    \label{ex_ineq4}
      \ln( t(n))
      \leqslant \ln K + (\ln( M) )(\ln (Mn))^{0.5}  \ln \alpha +
     \ln( \alpha) (\ln( n) )\left[\left(\ln (Mn))^{0.5}
       - (\ln n)^{0.5} \right) -\frac{2}{\ln( \alpha)} \right]
   \end{dmath*}

   Now if we let
   $B=\ln (Mn)^{0.5} - (\ln n)^{0.5}$ then
   $B(\ln (Mn)^{0.5} + (\ln n)^{0.5})
   =\ln (Mn)-(\ln n)=\ln M$ and so \[B=\frac{\ln M}{(\ln (Mn))^{0.5}
       + (\ln n)^{0.5}}.\]

   Thus
      \begin{dmath*}
    \label{ex_ineq5}
      \ln( t(n))
      \leqslant \ln K + (\ln( M) )(\ln (Mn))^{0.5} \ln \alpha +
     \ln( \alpha) (\ln( n) )\left[\left(\frac{\ln M}{(\ln (Mn))^{0.5}
       + (\ln n)^{0.5}}\right) -\frac{2}{\ln( \alpha)} \right]
   \end{dmath*}
        which becomes
         \begin{dmath*}
    \label{ex_ineq7}
      \ln( t(n))
      \leqslant \ln K +
     \ln \alpha \ln( M) \left[(\ln (Mn))^{0.5} +
     (\ln( n) )\left[\left(\frac{1}{(\ln (Mn))^{0.5}
       + (\ln n)^{0.5}}\right) -\frac{2}{\ln( \alpha)\ln M} \right]\right]
   \end{dmath*}
   which becomes
   \begin{dmath*}
    \label{ex_ineq8}
      \ln( t(n))
      \leqslant \ln K +
     \ln \alpha \ln( M) \left[\frac{(\ln (Mn))}{(\ln (Mn))^{0.5}} +
     (\ln( n) )\left[\left(\frac{1}{(\ln (Mn)^{0.5}
       + (\ln n)^{0.5}}\right) -\frac{2}{\ln( \alpha)\ln M} \right]\right]
   \end{dmath*}
   which becomes
     \begin{dmath*}
      \ln( t(n))
      \leqslant \ln K +
     \ln \alpha \ln( M) \left[\frac{(\ln (M)+\ln(n))}{(\ln (Mn))^{0.5}} +
     (\ln( n) )\left[\left(\frac{1}{(\ln (Mn)^{0.5}
       + (\ln n)^{0.5}}\right) -\frac{2}{\ln( \alpha)\ln M} \right]\right]
 \end{dmath*}
      which becomes
    \begin{dmath*}
      \ln( t(n))
      \leqslant \ln K +
     \ln \alpha \ln( M) \ln (n)
     \left[ \left[\frac{\ln (M)}{(\ln(n))\ln (Mn))^{0.5}} +
     \frac{1}{(\ln (Mn))^{0.5}} +
    \frac{1}{(\ln (Mn)^{0.5}
       + (\ln n)^{0.5}}\right] -\frac{2}{\ln( \alpha)\ln M} \right].
   \end{dmath*}

   The expression in the inside square brackets is going to  $0$ as $n\to\infty$, so eventually it will be less than $\frac{2}{\ln( \alpha)\ln M}$, contradicting the fact that $t\in\F$.
   \end{proof}

The next proposition proves that any function which is strongly-super-polynomial is also super-polynomial.

       \begin{proposition}
       \label{prop:strong_implies_super}
    Let $f \in \F$ be a non-zero function.
    If $f$ is strongly-super-polynomial, then
    $f$ is super-polynomial.
 \end{proposition}	

 \begin{proof}
    Since $f$ is strongly-super-polynomial, by Proposition~\ref{prop:strongSuper} we have $n^cf\ll f$ for {\em any} arbitrary $c>0$ we wish to choose.
    So for any $c>0$ there are positive constants
    $K_c,M_c$ and $N_c$ and
    an unbounded function $t_c \in \mathcal{F}$
    so that
    the inequality
    \begin{equation*}
    	n^c f(n) t_c(n) \leqslant K_c f(M_cn)
    \end{equation*} 	
    holds for all $n \geq N_c$.
    Therefore, for all $n \geq N_cM_c$ we have
    \begin{equation*}
    \label{str_sup_pol_eq1}
       \left\lfloor \frac{n}{M_c}
       \right\rfloor^c
       f\left( \left\lfloor \frac{n}{M_c}
       \right\rfloor\right)
       t_c\left(\left\lfloor \frac{n}{M_c}
       \right\rfloor\right) \leqslant K_c
       f \left( M_c \left\lfloor \frac{n}{M_c}
       \right\rfloor\right) \leqslant K_c f (n).
    \end{equation*}
    Taking the logarithm of this inequality
    we obtain
    \begin{equation*}
    \label{str_sup_pol_eq2}
       c \ln \left\lfloor \frac{n}{M_c}
       \right\rfloor +
       \ln f\left( \left\lfloor \frac{n}{M_c}
       \right\rfloor\right)  +  \ln t_c \left(\left\lfloor \frac{n}{M_c}
       \right\rfloor\right) \leqslant
       \ln K_c  + \ln f (n).
    \end{equation*}
    which becomes
        \begin{equation*}
    \label{str_sup_pol_eq3}
       c \ln \left\lfloor \frac{n}{M_c}
       \right\rfloor +
       \ln f\left( \left\lfloor \frac{n}{M_c}
       \right\rfloor\right)  +  \ln t_c \left(\left\lfloor \frac{n}{M_c}
       \right\rfloor\right) - \ln K_c\leqslant
         \ln f (n).
    \end{equation*}

    Now since $f$ and $t_c$ are both unbounded functions, there exists some  $N_c'\geq N_cM_c$ so that \[\ln f\left( \left\lfloor \frac{n}{M_c}
       \right\rfloor\right)  +  \ln t_c \left(\left\lfloor \frac{n}{M_c}
       \right\rfloor\right)-\ln K_c\geq 0\]
       for all $n\geq N_c'$. Thus we have
               \begin{equation*}
    \label{str_sup_pol_eq4}
       c \ln \left\lfloor \frac{n}{M_c}
       \right\rfloor
      \leq
         \ln f (n)
    \end{equation*} for all $n\geq N_c'$.

    Dividing sides
    by $\ln n$, we obtain that
         \begin{equation}
    \label{str_sup_pol_eq5}
       c \frac{\ln \left\lfloor \frac{n}{M_c}
       \right\rfloor}{\ln n}
       \leq
         \frac{\ln f (n)}{\ln n}
    \end{equation} for all $n\geq N_c'$.

    Now choose $I_c\in \N$ so that $I_c\geq N_c'$
    and $\frac{\ln \left\lfloor \frac{n}{M_c}
       \right\rfloor}{\ln n}>0.9$ for all $n\geq I_c$. Observe that the limit  $\lim_{n\to\infty}\frac{\ln \left\lfloor \frac{n}{M_c}
       \right\rfloor}{\ln n}=1$ and the term $\frac{\ln \left\lfloor \frac{i}{M_c}
       \right\rfloor}{\ln i}$ is increasing, so such a value $I_c$ exists.

    Then from this observation and equation~(\ref{str_sup_pol_eq5})  we get
      \begin{equation*}
    \label{str_sup_pol_eq6}
     0.9c\leq   c  \frac{\ln \left\lfloor \frac{n}{M_c}
       \right\rfloor}{\ln n}
       \leq
         \frac{\ln f (n)}{\ln n}
    \end{equation*} for all $n\geq I_c$.

    Since $c>0$ can be arbitrary, this shows that the limit of $  \frac{\ln f (n)}{\ln n}$ must go to $\infty$.
    \end{proof}

\end{appendix}

\section*{Acknowledgements}

The authors thank the anonymous reviewer for their helpful feedback and careful proofreading.
\bibliographystyle{plain}
\bibliography{BET_bibliography}

\end{document}

%% file: qi1.tex
\begin{tikzpicture}[scale=1.2]

 \begin{scope}[decoration={markings,mark = at position 0.6 with {\arrow[scale=2,black]{latex}}}]

  \draw[postaction={decorate}]
    (0,0) arc (320:360:-4 and 6.0);
        \draw  (-1.05,2.1) node {$w$};

       \draw
    (-.95,5.1) node[left=1pt] {$\rho^{-1}(\gamma)$};

    \draw[
    decorate, decoration=snake]
      (-.95,3.9)  .. controls (-.7,5)  .. (-.4,5.95) ;

   \fill[radius=1.5pt,blue]
    (0,0) circle node[left=1pt] {$1_G$};

   \fill[radius=1.5pt,blue]
    (-.95,3.9) circle node[left=1pt] {$\pi(w)=\pi(\rho(w))$};

   \fill[radius=1.5pt,blue]
    (-.4,5.95) circle node[above left=1pt] {$\psi(w)=\psi'(\rho(w))$};

     \end{scope}
\end{tikzpicture}

%% file: qi2.tex
\begin{tikzpicture}[scale=1.2]

 \begin{scope}[decoration={markings,mark = at position 0.6 with {\arrow[scale=2,black]{latex}}}]

   \draw
    (-.95,2.1) node[left=1pt] {$\rho(w)$};
       \draw
    (-.95,5.1) node[left=1pt] {$\gamma$};

    \draw[postaction={decorate}]
      (-.95,3.9)  .. controls (-.8,5)  .. (-.4,5.95) ;

         \draw[
         decorate,decoration=snake]
    (0,0) arc (320:360:-4 and 6.0);

   \fill[radius=1.5pt,blue]
    (0,0) circle node[left=1pt] {$1_G$};

   \fill[radius=1.5pt,blue]
    (-.95,3.9) circle node[left=1pt] {$\pi(w)=\pi(\rho(w))$};

   \fill[radius=1.5pt,blue]
    (-.4,5.95) circle node[above left=1pt] {$\psi(w)=\psi'(\rho(w))$};

     \end{scope}
\end{tikzpicture}

%% file: loop1.tex
\begin{tikzpicture}[scale=1.2]

 \begin{scope}[decoration={markings,mark = at position 0.5 with {\arrow[scale=2]{latex}}}]

   \draw
    (0,3) circle (3 and 3);

  \draw[postaction={decorate}]
    (0,0) arc (320:360:-4 and 6.0);
  \draw[postaction={decorate}]
    (0,0) arc (-90: 10:0.8 and 3.5);
        \draw  (-1,2) node {$u_i$};
            \draw  (1,2) node[right=-9pt]  {$u_{i+1}$};

         \draw[postaction={decorate}]  (0,6) node[above] {$s_i$};

    \draw[postaction={decorate}]
      (-.95,3.9)  .. controls (-.8,5)  .. (-.4,5.95) ;
    \draw  (-1,5) node {$\gamma_i$};
          \draw[postaction={decorate}]
        (.805,4.15)  .. controls (.7,5) .. (.4,5.95);
    \draw  (1,5) node[right=-9pt] {$\gamma_{i+1}$};

   \fill[radius=1.5pt,blue]
    (-.95,3.9) circle node[left=1pt] {$\pi(u_i)$};
       \fill[radius=1.5pt,blue]
    (.805,4.15) circle node[right=1pt] {$\pi(u_{i+1})$};

   \fill[radius=1.5pt,blue]
    (-.4,5.95) circle node[above left=1pt] {$g_i$};
       \fill[radius=1.5pt,blue]
    (.4,5.95) circle node[above right=1pt] {$g_{i+1}$};

     \end{scope}
\end{tikzpicture}

%% file: loop2.tex
\begin{tikzpicture}[scale=1.2]

 \begin{scope}[decoration={markings,mark = at position 0.5 with {\arrow[scale=2]{latex}}}]

   \draw
    (0,3) circle (3 and 3);

  \draw[postaction={decorate}]
    (0,0) arc (320:360:-4 and 6.0);
  \draw[postaction={decorate}]
    (0,0) arc (-90: 10:0.8 and 3.5);

       \draw[postaction={decorate}]  (0,6) node[above] {$s_i$};

    \draw[postaction={decorate}]
      (-.95,3.9)  .. controls (-.8,5)  .. (-.4,5.95) ;
    \draw  (-1,5) node {$\gamma_i$};
          \draw[postaction={decorate}]
        (.805,4.15)  .. controls (.7,5) .. (.4,5.95);
    \draw  (1,5) node[right=-9pt] {$\gamma_{i+1}$};

       \fill[radius=1.5pt,red]
    (-.4,3.3) circle node[ left=.5pt] {\footnotesize$\pi(x_{i,j})$};
       \fill[radius=1.5pt,red]
    (.4,3.4) circle node[ right=.5pt] {\footnotesize$\pi(x_{i+1,j})$};

       \fill[radius=1.5pt,red]
    (-.2,4.3) circle node[ above left=-.5pt and -.5pt] {\footnotesize$\psi(x_{i,j})$};
       \fill[radius=1.5pt,red]
    (.1,4.36) circle node[ above right=-.5pt and -.5pt] {\footnotesize$\psi(x_{i+1,j})$};

        \draw[red]
   (-.4,3.3) .. controls (-.3,3.7)  ..  (-.2,4.3) ;
          \draw[red]
         (.4,3.4) .. controls (.2,3.8) ..     (.1,4.36);

             \draw[red]  (-.25,3) node {\footnotesize$\beta_i$};
     \draw[red]  (.35,3) node {\footnotesize$\beta_{i+1}$};

                \draw[red] (-.2,4.3) parabola (.1,4.36);

             \draw[red,postaction={decorate}]     (-.84,2.6) parabola   (-.4,3.3);
                  \draw[red,postaction={decorate}]   (.8,2.7) parabola       (.4,3.4);

                         \fill[radius=1.5pt,blue]
    (-.84,2.6) circle node[ left=1pt] {$p_{i,j}$};
       \fill[radius=1.5pt,blue]
    (.8,2.7) circle node[right=1pt] {$p_{i+1,j}$};

   \fill[radius=1.5pt,blue]
    (-.95,3.9) circle node[left=1pt] {$\pi(u_i)$};
       \fill[radius=1.5pt,blue]
    (.805,4.15) circle node[right=1pt] {$\pi(u_{i+1})$};

   \fill[radius=1.5pt,blue]
    (-.4,5.95) circle node[above left=1pt] {$g_i$};
       \fill[radius=1.5pt,blue]
    (.4,5.95) circle node[above right=1pt] {$g_{i+1}$};

   \fill[radius=1.0pt,blue]
    (-.4,5.95) circle node[above left=1pt] {$g_i$};
       \fill[radius=1.0pt,blue]
    (.4,5.95) circle node[above right=1pt] {$g_{i+1}$};

     \end{scope}
\end{tikzpicture}

%% file: loop3.tex
\begin{tikzpicture}[scale=1.2]

 \begin{scope}[decoration={markings,mark = at position 0.5 with {\arrow[scale=2]{latex}}}]

   \draw
    (0,3) circle (3 and 3); 

  \draw
    (0,0) arc (320:360:-4 and 6.0);
  \draw
    (0,0) arc (-90: 10:0.8 and 3.5);

       \draw[postaction={decorate}] (0,6) node[above] {$s_i$};

    \draw[postaction={decorate}]
      (-.95,3.9)  .. controls (-.8,5)  .. (-.4,5.95) ;
    \draw  (-1,5) node {\footnotesize$\gamma_i$};
          \draw[postaction={decorate}]
        (.805,4.15)  .. controls (.7,5) .. (.4,5.95);
    \draw  (1,5) node[right=-9pt] {\footnotesize$\gamma_{i+1}$};

       \fill[radius=1.0pt,red]
    (-.4,3.3) circle ;
       \fill[radius=1.0pt,red]
    (.4,3.4) circle ;

       \fill[radius=1.0pt,red]
    (-.2,4.3) circle; 
       \fill[radius=1.0pt,red]
    (.1,4.36) circle ; 

        \draw[red]
   (-.4,3.3) .. controls (-.3,3.7)  ..  (-.2,4.3) ;
          \draw[red]
         (.4,3.4) .. controls (.2,3.8) ..     (.1,4.36);

                \draw[red] (-.2,4.3) parabola (.1,4.36);
             \draw[red,postaction={decorate}]     (-.84,2.6) parabola   (-.4,3.3);
                  \draw[red,postaction={decorate}]   (.8,2.7) parabola       (.4,3.4);

%
%


        \draw[red]
   (-.5,3.7) .. controls (-.4,4.1)  ..  (-.2,4.7) ;
          \draw[red]
         (.4,3.8) .. controls (.2,4.3) ..     (.1,4.76);

                \draw[red] (-.2,4.7) parabola (.1,4.76);
             \draw[red]    (-.9,3) parabola   (-.5,3.7);
                  \draw[red]   (.8,3.1) parabola       (.4,3.8);


        \draw[red]
(-.6,4.2) .. controls (-.4,4.7)  ..     (-.2,5.1) ;
          \draw[red]
         (.5,4.3) .. controls (.2,4.8) ..      (.1,5.15);

                \draw[red] (-.2,5.1) parabola (.1,5.15); 
             \draw[red]    (-.92,3.45) parabola   (-.6,4.2) ; 
                  \draw[red]   (.8,3.6) parabola     (.5,4.3) ; 

        \draw[red]
    (-.4,2.7) .. controls (-.4,2.7)  ..     (-.18,3.6) ;
          \draw[red]
           (.4,2.7).. controls (.2,3) ..         (.1,3.65);

                \draw[red]  (-.18,3.6) parabola     (.1,3.65); 
             \draw[red]      (-.78,2.2)  parabola    (-.4,2.7) ; 
                  \draw[red]      (.74,2.3) parabola      (.4,2.7) ; 


        \draw[red]
  (-.34,2.2) .. controls (-.3,2.4)  ..     (-.18,3) ;
          \draw[red]
         (.3,2.3) .. controls (.2,2.5) ..         (.1,3.05);

                \draw[red]  (-.18,3) parabola     (.1,3.05); 
             \draw[red]      (-.7,1.8)  parabola    (-.34,2.2); 
                  \draw[red]       (.7,1.9) parabola          (.3,2.3)  ; 

        \draw[red]
    (-.3,1.8)  .. controls (-.2,2.1)  ..        (-.18,2.4) ;
          \draw[red]
           (.25,1.9)  .. controls (.18,2.1) ..           (.1,2.45);

                \draw[red]     (-.18,2.4) parabola         (.1,2.45); 
             \draw[red]     (-.56,1.3)  parabola        (-.3,1.8) ; 
                  \draw[red]        (.65,1.4) parabola             (.25,1.9)  ; 

        \draw[red]
 (-.2,1.35) .. controls (-.15,1.5)  ..           (-.14,1.95)  ;
          \draw[red]
    (.2,1.4)   .. controls (.15,1.6) ..              (.07,2);

                \draw[red]     (-.14,1.95)  parabola      (.07,2); 
             \draw[red]       (-.4,.85)  parabola       (-.2,1.35); 
                  \draw[red]         (.54,.9) parabola                (.2,1.4)  ; 

           \fill[radius=1.0pt,blue]
    (-.9,3) circle;
       \fill[radius=1.0pt,blue]
    (.8,3.1) circle ;

       \fill[radius=1.0pt,red]
    (-.5,3.7) circle ;
       \fill[radius=1.0pt,red]
    (.4,3.8) circle ;

       \fill[radius=1.0pt,red]
    (-.2,4.7) circle ;
       \fill[radius=1.0pt,red]
    (.1,4.76) circle ;

           \fill[radius=1.0pt,blue]
    (-.92,3.45) circle; 
       \fill[radius=1.0pt,blue]
    (.8,3.6) circle ; 

       \fill[radius=1.0pt,red]
    (-.6,4.2) circle ; 
       \fill[radius=1.0pt,red]
    (.5,4.3) circle ; 

       \fill[radius=1.0pt,red]
    (-.2,5.1) circle ; 
       \fill[radius=1.0pt,red]
    (.1,5.15) circle ; 

           \fill[radius=1.0pt,blue]
    (-.78,2.2) circle; 
       \fill[radius=1.0pt,blue]
    (.74,2.3) circle ; 

       \fill[radius=1.0pt,red]
    (-.4,2.7) circle ; 
       \fill[radius=1.0pt,red]
    (.4,2.7) circle ; 

       \fill[radius=1.0pt,red]
    (-.18,3.6) circle ; 
       \fill[radius=1.0pt,red]
    (.1,3.65) circle ; 

           \fill[radius=1.0pt,blue]
    (-.7,1.8) circle; 
       \fill[radius=1.0pt,blue]
    (.7,1.9) circle ; 

       \fill[radius=1.0pt,red]
    (-.34,2.2) circle ; 
       \fill[radius=1.0pt,red]
    (.3,2.3) circle ; 

       \fill[radius=1.0pt,red]
    (-.18,3) circle ; 
       \fill[radius=1.0pt,red]
    (.1,3.05) circle ; 

           \fill[radius=1.0pt,blue]
    (-.4,.85) circle; 
       \fill[radius=1.0pt,blue]
    (.54,.9) circle ; 

       \fill[radius=1.0pt,red]
    (-.2,1.35) circle ; 
       \fill[radius=1.0pt,red]
    (.2,1.4) circle ; 

       \fill[radius=1.0pt,red]
    (-.14,1.95) circle ; 
       \fill[radius=1.0pt,red]
    (.07,2) circle ; 

           \fill[radius=1.0pt,blue]
    (-.56,1.3) circle; 
       \fill[radius=1.0pt,blue]
    (.65,1.4) circle ; 

       \fill[radius=1.0pt,red]
    (-.3,1.8) circle ; 
       \fill[radius=1.0pt,red]
    (.25,1.9) circle ; 

       \fill[radius=1.0pt,red]
    (-.18,2.4) circle ; 
       \fill[radius=1.0pt,red]
    (.1,2.45) circle ; 

   \fill[radius=1.0pt,blue]
    (-.4,5.95) circle node[above left=1pt] {\footnotesize$g_i$};
       \fill[radius=1.0pt,blue]
    (.4,5.95) circle node[above right=1pt] {\footnotesize$g_{i+1}$};

       \fill[radius=1.0pt,blue]
    (-.84,2.6) circle node[ left=1pt] {\footnotesize$p_{i,j}$};
       \fill[radius=1.0pt,blue]
    (.8,2.7) circle node[right=1pt] {\footnotesize$p_{i+1,j}$};

   \fill[radius=1.0pt,blue]
    (-.95,3.9) circle node[left=1pt] {\footnotesize$\pi(u_i)$};
       \fill[radius=1.0pt,blue]
    (.805,4.15) circle node[right=1pt] {\footnotesize$\pi(u_{i+1})$};

   \fill[radius=1.0pt,blue]
    (-.4,5.95) circle node[above left=1pt] {\footnotesize$g_i$};
       \fill[radius=1.0pt,blue]
    (.4,5.95) circle node[above right=1pt] {\footnotesize$g_{i+1}$};

     \end{scope}
\end{tikzpicture}